\documentclass[graybox,envcountsect,envcountsame]{svmult}

\usepackage{mathptmx}       % selects Times Roman as basic font
\usepackage{helvet}         % selects Helvetica as sans-serif font
\usepackage{courier}        % selects Courier as typewriter font
\usepackage{type1cm}        % activate if the above 3 fonts are
\usepackage{makeidx}         % allows index generation
\usepackage{graphicx}        % standard LaTeX graphics tool
\usepackage{multicol}        % used for the two-column index
\usepackage[bottom]{footmisc}% places footnotes at page bottom

\makeindex             % used for the subject index
\usepackage{amsmath,amssymb,amsxtra,verbatim}
\usepackage{latexsym}
\usepackage{color}
\DeclareMathOperator{\gen}{gen}
\DeclareMathOperator{\Gr}{Gr}
\DeclareMathOperator{\diag}{diag}

\DeclareMathOperator{\Hom}{Hom}

\DeclareMathOperator{\Gal}{Gal}

\DeclareMathOperator{\Cl}{Cl}

\DeclareMathOperator{\PGSp}{PGSp}

\DeclareMathOperator{\GL}{GL}

\DeclareMathOperator{\PGL}{PGL}

\DeclareMathOperator{\tr}{tr}

\DeclareMathOperator{\Coind}{Coind}

\newcommand{\lra}{\longrightarrow}

\newcommand{\ZZ}{\mathbb{Z}}
\newcommand{\QQ}{\mathbb{Q}}

\newcommand{\RR}{\mathbb{R}}

\newcommand{\FF}{\mathbb{F}}

\newcommand{\Z}{\mathbb{Z}}
\newcommand{\Q}{\mathbb{Q}}
\newcommand{\C}{\mathbb{C}}
\newcommand{\R}{\mathbb{R}}
\newcommand{\F}{\mathbb{F}}

\newcommand{\cH}{\mathcal{H}}

\newcommand{\fA}{\mathfrak{A}}

\newcommand{\fB}{\mathfrak{B}}

\newcommand{\fd}{\mathfrak{d}}

\newcommand{\fp}{\mathfrak{p}}
\newcommand{\fP}{\mathfrak{P}}
\newcommand{\fq}{\mathfrak{q}}
\newcommand{\fQ}{\mathfrak{Q}}

\renewcommand{\phi}{\varphi}
\renewcommand{\epsilon}{\varepsilon}

\DeclareMathOperator{\Aut}{Aut}

\DeclareMathOperator{\opchar}{char}

\newcommand{\calH}{\mathcal H}

\newcommand{\calO}{\mathcal O}

\newcommand{\bdA}{\mathbf A}

\newcommand{\sfA}{\mathsf A}
\newcommand{\sfG}{\mathsf G}
\newcommand{\sfU}{\mathsf U}

\newcommand{\Fhat}{\widehat{F}}
\newcommand{\Ghat}{\widehat{G}}
\newcommand{\Khat}{\widehat{K}}

\newcommand{\Lambdahat}{\widehat{\Lambda}}
\newcommand{\ghat}{\widehat{g}}
\newcommand{\phat}{\widehat{p}}
\newcommand{\uhat}{\widehat{u}}
\newcommand{\xhat}{\widehat{x}}
\newcommand{\Qhat}{\widehat{\mathbb Q}}
\newcommand{\Zhat}{\widehat{\mathbb Z}}

\newenvironment{enumalg}
{\begin{enumerate}}
{\end{enumerate}}

\newenvironment{enumalgalph1}
{\begin{enumerate}}
{\end{enumerate}}

\newenvironment{enumalgalph}
{\begin{enumerate}}
{\end{enumerate}}

\newenvironment{enumroman}
{\begin{enumerate}}
{\end{enumerate}}

\begin{document}

\title*{Lattice methods for algebraic modular forms on classical groups}
\author{Matthew Greenberg and John Voight}
\institute{Matthew Greenberg \at University of Calgary, 2500 University Drive NW, Calgary, AB, T2N 1N4, Canada, \email{mgreenbe@math.ucalgary.ca}
\and John Voight \at Department of Mathematics and Statistics, University of Vermont, 16 Colchester Ave, Burlington, VT 05401, USA, \email{jvoight@gmail.com}}
\maketitle

\abstract*{We use Kneser's neighbor method and isometry testing for lattices due to Plesken and Souveigner to compute systems of Hecke eigenvalues associated to definite forms of classical reductive algebraic groups.}

\abstract{We use Kneser's neighbor method and isometry testing for lattices due to Plesken and Souveigner to compute systems of Hecke eigenvalues associated to definite forms of classical reductive algebraic groups.}

% Questions for AIM:
% Is there a Jacquet-Langlands correspondence for orthogonal and unitary groups between two different quadratic spaces?  There is probably something in terms of a theta correspondence, but I wonder if it is really that explicit.  We could build up some really nice careful data behind it!
% What are the exact relations between the mass, the class number, the sizes of the orthogonal groups, and the diameter of the neighbour graph on isometry classes?

\section{Introduction}

Let $Q(x)=Q(x_1,\dots,x_n) \in \Z[x_1,\dots,x_n]$ be an even positive definite integral quadratic form in $n$ variables with discriminant $N$.  A subject of extensive classical study, continuing today, concerns the number of representations of an integer by the quadratic form $Q$.  To do so, we form the corresponding generating series, called the \emph{theta series} of $Q$:
\[ \theta_Q(q)=\sum_{x \in \Z^n} q^{Q(x)} \in \Z[[q]]. \]
By letting $q=e^{2\pi iz}$ for $z$ in the upper half-plane $\calH$, we obtain a holomorphic function $\theta:\calH \to \C$;  owing to its symmetric description, this function is a classical modular form of weight $n/2$ and level $4N$.  For example, in this way one can study the representations of an integer as the sum of squares via Eisenstein series for small even values of $n$.

Conversely, theta series can be used to understand spaces of classical modular forms.  This method goes by the name \emph{Brandt matrices} as it goes back to early work of Brandt and Eichler \cite{Eichlerbasis,Eichlerbasis1} (the \emph{basis problem}).  From the start, Brandt matrices were used to computationally study spaces of modular forms, and explicit algorithms were exhibited by Pizer \cite{Pizer}, Hijikata, Pizer, and Shemanske \cite{HPS}, and Kohel \cite{kohel}.  In this approach, a basis for $S_2(N)$ is obtained by linear combinations of theta series associated to (right) ideals in a quaternion order of discriminant $N$; the Brandt matrices which represent the action of the Hecke operators are obtained via the combinatorial data encoded in the coefficients of theta series.  These methods have also been extended to Hilbert modular forms over totally real fields, by Socrates and Whitehouse \cite{SocratesWhitehouse}, Demb\'el\'e \cite{Dembele}, and Demb\'el\'e and Donnelly \cite{DD}.

The connection between such arithmetically-defined counting functions and modular forms is one piece of the Langlands philosophy, which predicts deep connections between automorphic forms in different guises via their Galois representations.  In this article, we consider algorithms for computing systems of Hecke eigenvalues in the more general setting of algebraic modular forms, as introduced by Gross \cite{Gross1}.  Let $\sfG$ be a \emph{linear algebraic group} defined over $\Q$, a closed algebraic subgroup of the algebraic group $\GL_n$.  (For simplicity now we work over $\Q$, but in the body we work with a group $\sfG$ defined over a number field $F$; to reduce to this case, one may just take the restriction of scalars.)  Let $\sfG(\Z)=\sfG(\Q) \cap \GL_n(\Z)$ be the group of integral points of $\sfG$.  

Suppose that $\sfG$ is \emph{connected} as an algebraic variety and \emph{reductive}, so that its maximal connected unipotent normal subgroup is trivial (a technical condition important for the theory).  Let $G_\infty=\sfG(\R)$ denote the real points of $\sfG$.  Then $G_\infty$ is a real Lie group with finitely many connected components.  

Now we make an important assumption that allows us to compute via arithmetic and lattice methods: we suppose that $G_\infty$ is compact.  For example, we may take $\sfG$ to be a special orthogonal group, those transformations of determinant $1$ preserving a positive definite quadratic form over a totally real field, or a unitary group, those preserving a definite Hermitian form relative to a CM extension of number fields.  Under this hypothesis, Gross \cite{Gross1} showed that automorphic forms arise without analytic hypotheses and so are called \emph{algebraic modular forms}.

Let $\Qhat=\Q \otimes_{\Z} \Zhat$ be the finite adeles of $\Q$.  Let $\Khat$ be a compact open subgroup of $\Ghat=\sfG(\Qhat)$ (a choice of \emph{level}), let $G=\sfG(\Q)$, and let 
\[ Y = G \backslash \Ghat/\Khat. \]  
The set $Y$ is finite.
%\todo{I like to call it $Y$ because it is like the (open) modular curve!  $D$ is also natural for the ``disc'' we are working in, but I don't think we will really need separate notation for that (and letters are already scarce!)}
%Sounds reasonable.
Let $W$ be an irreducible (finite-dimensional) representation of $G$.  Then the space of \emph{modular forms for $\sfG$ of weight $W$ and level $\Khat$} is
\[ M(W,\Khat) = \{f : \Ghat/\Khat \to W \mid f(\gamma g)=\gamma f(g) \text{ for all $\gamma \in G$}\}. \]
Such a function $f \in M(W,\Khat)$ is determined by its values on the finite set $Y$; indeed, if $W$ is the trivial representation, then modular forms are simply functions on $Y$.  The space $M(W,\Khat)$ is equipped with an action of \emph{Hecke operators} for each double coset $\Khat \phat \Khat$ with $\phat \in \Ghat$; these operators form a ring under convolution, called the \emph{Hecke algebra}.

Algebraic modular forms in the guise of Brandt matrices and theta series of quaternary quadratic forms, mentioned above, correspond to the case where $G=\PGL_1(B) = B^\times/F^\times$ where $B$ is a definite quaternion algebra over a totally real field $F$.  The first more general algorithmic consideration of algebraic modular forms was undertaken by Lanksy and Pollack \cite{LanskyPollack}, who computed with the group $\sfG=\PGSp_4$ and the exceptional group $\sfG=G_2$ over $\Q$.  Cunningham and Demb\'el\'e \cite{CD} later computed Siegel modular forms over totally real fields using algebraic modular forms, and Loeffler \cite{Loeffler} has performed computations with the unitary group $U(2)$ relative to the imaginary quadratic extension $\Q(\sqrt{-11})/\Q$ and $U(3)$ relative to $\Q(\sqrt{-7})/\Q$.  
%\todo{Schulze-Pillot, Nebe, and Scharlau have also computed Siegel modular forms using Hecke operators on lattices.}
%If I recall, what they did was use Schiemann's lattice enumeration code to compute the theta functions associated to representative lattices for classes in a given genus, letting them compute the subspace of Siegel modular forms spanned by theta functions.  They compute the Hecke operators on the associated Siegel modular forms using formulas for these on q-expansions.  So the spirit of these computations is somewhat different...
In this paper, we consider the case where the group $\sfG$ arises from a definite special orthogonal or unitary group.  Our main idea is the use lattice methods, making these computations efficient.  This connection is undoubtedly known to the experts, and our small contribution is make it explicit and discuss the relevant computational aspects.  
%\todo{I remember Gabi Nebe was working on this last year; probably she has made some progress by now.  Speaking of, do we want to include her as an organizer?  I think she will be able to help us sort out the right German lattice consortium invitees better than we could.}  
%Yes!
We conjecture that, assuming an appropriate analogue of the Ramanujan-Petersson conjecture, lattice methods will run in polynomial time in the output size.  (This is known to be true for Brandt matrices, by work of Kirschmer and the second author \cite{KV}.)
% To compute algebraic modular forms, we develop two key methods.  First, we need an \emph{enumeration algorithm} for representatives $\{\xhat_1 \Khat, \cdots, \xhat_h \Khat\}$ of the set $Y$; second, we need a \emph{reduction algorithm} to write $\xhat_i \phat_j \Khat = \gamma_{ij} \xhat_{i^*} \Khat$ with $\gamma_{ij} \in G$.  We recast both of these problems in terms of lattices: the first consists of enumerating the genus of a quadratic form, and the second concerns isomorphism test

To illustrate our method as we began, let $Q$ be a positive definite quadratic form in $d$ variables over a totally real field $F$, and let $\sfG=SO(Q)$ be the special orthogonal group of $Q$ over $F$.  (To work instead with unitary groups, we simply work with a Hermitian form instead.)  Then $\sfG$ is a connected 
%\todo{The orthogonal group is not connected!} 
%To true!  Nor is it simply connected, a fact which is giving me grief at the moment.
reductive group with $G_\infty=\sfG(F \otimes_{\Q} \R)$ compact.  Let $\Lambda$ be a $\Z_F$-lattice in $F^d$.  
%\todo{I hope you like this choice of notation.  It frees up $L$ to be used for a CM field, which is what Gross uses, and when mathematical objects become `integral' I think the notation should become more `complicated' to signify this.}
%I like it just fine.
Then the stabilizer $\Khat \subset \Ghat$ of $\Lambdahat=\Lambda \otimes_\Z \Qhat$ is an open compact subgroup and the set $Y=G\backslash \Ghat/\Khat$ is in natural bijection with the finite set of equivalence classes of lattices in the \emph{genus} of $\Lambda$, the set of lattices which are \emph{locally equivalent} to $\Lambda$.  

The enumeration of representatives of the genus of a lattice has been studied in great detail; we use Kneser's neighbor method \cite{Kneser}.  (See the beginning of Section 5 for further references to the use of this method.)
%\todo{What other cites are needed here?}  
%I added Schiemann.
Let $\fp \subset \Z_F$ be a nonzero prime ideal with residue class field $\F_\fp$.  We say that two $\Z_F$-lattices $\Lambda,\Pi \subset F^n$ are \emph{$\fp$-neighbors} if we have $\fp \Lambda, \fp \Pi \subset \Lambda \cap \Pi$ and
\[ \dim_{\F_\fp} \Lambda/(\Lambda \cap \Pi) = \dim_{\F_\fp} \Pi/(\Lambda \cap \Pi) = 1. \]
%\todo{It was potentially confusing to use $M$ for a lattice, since $M$ was already a space of modular forms.  We could still use it here (Mu comes after Lambda); I chose $\Pi$ since it is the next simplest capital Greek letter, and because $\Pi$ makes me think $\fp$-neighbor.}  
% Okay.  Is M the upper case version of mu?
The $\fp$-neighbors of $\Lambda$ are easy to construct, are locally equivalent to $\Lambda$, and by strong approximation, every class in the genus is represented by a $\fp$-neighbor for some $\fp$.  In fact, by the theory of elementary divisors, the Hecke operators are also obtained as a summation over $\fp$-neighbors.  Therefore the algorithmic theory of lattices is armed and ready for application to computing automorphic forms.

% A general method for computing algebraic modular forms was first investigated by Lansky and Pollack \cite{LanskyPollack}.  In this paper, we consider the case of certain classical groups where the algorithmic tasks involved are tractable using lattice techniques.  Concretely, $\sfG$ will be either the orthogonal group of a positive-definite quadratic form or the unitary group of a positive-definite Hermitian form, and $\Khat$ is the stabilizer of an integral lattice.  % Further, we will compute Hecke operators only at places $\fp$ where $K_\fp$ a hyperspecial maximal compact.  

The main workhorse in using $\fp$-neighbors in this way is an algorithm for isometry testing between lattices (orthogonal, Hermitian, or otherwise preserving a quadratic form).  For this, we rely on the algorithm of Plesken and Souvignier \cite{PS}, which matches up short vectors and uses other tricks to rule out isometry as early as possible. This algorithm was implemented in \textsf{Magma} \cite{Magma} by Souvignier, with further refinements to the code contributed by Steel, Nebe, and others.  

These methods also apply to compact forms of symplectic groups; see Chisholm \cite{Chisholm}.  We anticipate that these methods can be generalized to a wider class of reductive groups, and believe that such an investigation would prove valuable for explicit investigations in the Langlands program.

% As a final application, we study Galois representations, and exhibit number fields with small ramification and large (nonsolvable) Galois group.  Already the case of a (positive definite) quadratic form in many variables over $\Q$ of discriminant $1$ we expect will exhibit many interesting Galois representations.

The outline of this paper is as follows.  In Section 2, we give basic terminology and notation for algebraic modular forms.  In section 3, we review orthogonal and unitary groups and their Hecke theory.  In section 4 we discuss elementary divisors in preparation for section 5, where we give an exposition of Kneser's neighbor method and translate Hecke theory to the lattice setting.  In section 6, we present the algorithm, and we conclude in section 7 with some explicit examples.

% \todo{People we should send the paper to: Clifton Cunningham, David Loeffler, David Pollack, Lassina Demb\'el\'e, Thomas Shemanske, Gabriele Nebe, Rainer Schulze-Pillot, Winfried Scharlau, Billy Chan...}

%  \todo{There's one other thing that's been on my mind.  Since $\GL_n$ embeds inside the unitary group at a split place, it seems like we're going to be able to see essentially all Galois representations of a totally real field this way!  I mean, take any finite group, embed it into $\GL_n$ over a finite field which is a residue field of $F$, embed this into a unitary group with respect to a CM extension of $F$ where the prime splits; then we should be seeing it in algebraic modular forms for the unitary group modulo $\fp$.  This is pretty intense!

% Then we start thinking about modularity lifting in this context when you get something big enough in $\GL_n$.  But more immediately, shouldn't that mean we should be able to find $\GL_n(\F_2)$ extensions of $\Q$ ramified only at $2$ for large enough $n$?  Do you know when the Galois representations exist?  I guess I'm still waiting to try this out with an even unimodular lattice over $\Q$, so we get $SO_n(\F_2)$-extensions, or do those always split up...?}

\section{Algebraic modular forms}

In this first section, we define algebraic modular forms; a reference is the original work of Gross \cite{Gross1}.  
%\todo{I made a bunch of notation changes to what you wrote.  For example, Gross lists the weight first, then the level.  I'm willing to negotiate on almost all of them!}
%I'm willing to surrender on most of them!

% \todo{To do: coordinate notation, etc.\ with Loeffler.}

\subsection*{Algebraic modular forms}

Let $F$ be a totally real number field and let 
\[
F_\infty=F\otimes_\QQ\RR\cong\RR^{[F:\Q]}.
\]
Let $\Qhat=\Q \otimes_{\Z} \Zhat$ be the finite adeles of $\Q$, let $\Fhat = F \otimes_{\Q} \Qhat$ be the ring of finite adeles of $F$.

Let $\sfG$ be a connected, reductive algebraic group over $F$.  We make the important and nontrivial assumption that the Lie group $G_\infty = G(F_\infty)$ is compact.  Let $\Ghat=\sfG(\Fhat)$ and $G=\sfG(F)$.  

Let $\rho:G \to W$ be an irreducible (finite-dimensional) representation of $G$ defined over a number field $E$.  % \todo{If we follow Gross, we could take only those defined over $\Q$, and then we end up taking the product of Galois conjugates as in \S 3 of Gross.  Is that what we want?  In the case of Hilbert modular forms, we work with representations over $\C$ and then ultimately a number field.  I don't see any harm in allowing our representation to be defined over a number field $E$.}
% \todo{Why do people refer to ``algebraic'' representations in this context?}
% \todo{Finally, why does the representation need to be irreducible?  I don't see why that's necessary anywhere.}

\begin{definition}
The space of \emph{algebraic modular forms for $G$ of weight $W$} is 
\[
M(\sfG,W) = \left\{f : \Ghat \to W\ \Big|\ \begin{minipage}{39ex} 
$f$ is locally constant and \\
$f(\gamma \ghat)=\gamma f(\ghat)$ for all $\gamma \in G$ and $\ghat \in \Ghat$
\end{minipage}
\right\}. \]
\end{definition}

We will often abbreviate $M(W)=M(\sfG,W)$.  

% \todo{Gross doesn't remember the group in the notation, and since we aren't moving between groups, I think we can omit it.}

% \todo{We're losing a little bit of generality with this definition, because we are not allowing interesting open compact subgroups of $G_\infty$.  (Gross takes locally constant functions on $\sfG(\bdA)$, not $\Ghat$.)  In our case, I don't think it makes much difference, and so I'd rather stick with this.  But then we should probably make a remark to this effect, since it is as though we are presenting the general theory here.}

%\todo{Do we want the representation to be a left $G$-module and not a right $G$-module, like we chose for Hilbert modular forms?  *ducks*  If we do, we have to act by the inverse on the right; ugly.}
%Let's try to leave it as is for the moment.   I'm hating on Magma for all their modules being right modules.

% \todo{There is little point in naming the representation $\rho$, since it just adds to the notation; or should we be explicit about that in the definition?  It seems cumbersome to write $\rho(\gamma)f(\ghat)$.}

Each $f \in M(W)$ is constant on the cosets of a compact open subgroup $\Khat \subset \Ghat$, so $M(W)$ is the direct limit of the spaces
\begin{equation} \label{MVK}
M(W,\Khat) = \left\{f : \Ghat \to W \ \Big|\ \begin{minipage}{28ex} \begin{center}
$f(\gamma \ghat\uhat)=\gamma f(\ghat)$ \\ 
for all $\gamma \in G$, $\ghat \in \Ghat$, $\uhat \in \Khat$ 
\end{center}
\end{minipage}
\right\}.
\end{equation}
of modular forms of \emph{level $\Khat$}.  We will consider these smaller spaces, so let $\Khat \subset \Ghat$ be an open compact subgroup.  When $W=E$ is the trivial representation, $M(W,\Khat)$ is simply the space of $E$-valued functions on the space $Y=G \backslash \Ghat / \Khat$.

\begin{proposition}[{\cite[Proposition 4.3]{Gross1}}] \label{Yisfinite}
The set $Y=G\backslash \Ghat/\Khat$ is finite.
\end{proposition}

Let $h=\# Y$.  Writing
\begin{equation} \label{GGK}
\Ghat=\bigsqcup_{i=1}^h G \widehat{x}_i \Khat,
\end{equation}
it follows from the definition that any $f\in M(W,\Khat)$ is completely determined by the elements $f(\xhat_i)$ with $i=1,\ldots,h$.  Let
\[
\Gamma_i = G \cap \xhat_i \Khat \xhat_i^{-1}.
\]
The (arithmetic) group $\Gamma_i$, as a discrete subgroup of the compact group $G_\infty$, is finite \cite[Proposition 1.4]{Gross1}.  

\begin{lemma} \label{finitedim}
The map
\begin{align*}
M(W,\Khat) & \lra \bigoplus_{i=1}^h H^0(\Gamma_i,W) \\
f &\mapsto (f(\xhat_1),\ldots,f(\xhat_h))
\end{align*}
is an isomorphism of $F$-vector spaces, where
\[ H^0(\Gamma_i,W)=\{v \in W : \gamma v = v \text{ for all $\gamma \in \Gamma_i$}\}.  \]  
\end{lemma}
%\todo{If we write $W^{\Gamma_i}$, I fear that could get confused with a right action...}
%and I dislike left-superscripts.
In particular, from Lemma \ref{finitedim} we see that $M(W,\Khat)$ is finite-dimensional as an $E$-vector space.

\subsection*{Hecke operators}

The space $M(W,\Khat)$ comes equipped with the action of Hecke operators, defined as follows.  Let $\cH(\Ghat,\Khat)=\cH(\Khat)$ be the space of locally constant, compactly supported, $\Khat$-bi-invariant functions on $\Ghat$.  Then $\cH(\Ghat,\Khat)$ is a ring under convolution, called the \emph{Hecke algebra}, and is generated by the characteristic functions $T(\phat)$ of double cosets $\Khat \phat \Khat$ for $\phat \in \Ghat$.  Given such a characteristic function $T(\phat)$, decompose the double coset $\Khat \phat \Khat$ into a disjoint union of right cosets
\begin{equation} \label{KpKdoubletoright}
\Khat \phat \Khat = \bigsqcup_j \phat_j \Khat
\end{equation}
and define the action of $T(\phat)$ on $f\in M(W,\Khat)$ by
\begin{equation} \label{hecke1}
(T(\phat) f)(\ghat)=\sum_j f(\ghat \phat_j).
\end{equation}
This action is well-defined (independent of the choice of representative $\phat$ and representatives $\phat_j$) by the right $\Khat$-invariance of $f$.  Finally, a straightforward calculation shows that the map in Lemma \ref{finitedim} is Hecke equivariant.

\subsection*{Level}

There is a natural map which relates modular forms of higher level to those of lower level by modifying the coefficient module, as follows \cite[\S 8]{DV}.  Suppose that $\Khat' \leq \Khat$ is a finite index subgroup.  Decomposing as in (\ref{GGK}), we obtain a bijection
\begin{align*}
G\backslash \Ghat/\Khat' = \bigsqcup_{i=1}^h G \backslash \bigl( G \xhat_i \Khat \bigr) / \Khat'
&\xrightarrow{\sim} \bigsqcup_{i=1}^h \Gamma_i \backslash \Khat_i / \Khat_i' \\
G(\gamma \xhat_i \uhat)\Khat' &\mapsto \Gamma_i(\xhat_i \uhat \xhat_i^{-1})\Khat_i'
\end{align*}
for $\gamma \in G$ and $\uhat \in \Khat$.  This yields 
\[ M(W,\Khat') \xrightarrow{\sim} H^0(\Gamma_i, \Hom(\Khat_i/\Khat_i', W)) \cong
\bigoplus_{i=1}^h H^0(\Gamma_i, \Coind_{\Khat_i'}^{\Khat_i} W). \]
Via the obvious bijection 
\begin{equation} \label{KiK}
\Khat_i/\Khat_i' \cong \Khat/\Khat', 
\end{equation} 
letting $W= \Coind_{\Khat'}^{\Khat} W$ we can also write
\begin{equation} \label{MVKp}
M(W,\Khat') \cong \bigoplus_{i=1}^h H^0(\Gamma_i, W_i)
\end{equation}
where $W_i$ is the representation $W$ with action twisted by the identification (\ref{KiK}).  Moreover, writing $\Khat=(K_\fp)_\fp$ in terms of its local components, for any Hecke operator $T(\phat)$ such that 
\begin{center}
if $\phat \not\in K'_\fp$ then $K_\fp=K_\fp'$
\end{center}
(noting that $\phat \in K'_\fp$ for all but finitely many primes $\fp$), the same definition (\ref{hecke1}) applies and by our hypothesis we have a simultaneous double coset decomposition 
\[ \Khat' \phat \Khat' = \bigsqcup_j \phat_j \Khat' \quad\text{and}\quad \Khat \phat \Khat = \bigsqcup_j \phat_j \Khat. \]
%\todo{In general, for the other Hecke operators you end up with all kinds of trouble, and even computing the Atkin-Lehner map in this way is not as straightforward as you may hope.}  
%And I haven't bothered coding them!
Now, comparing (\ref{MVKp}) to the result of Lemma \ref{finitedim}, we see in both cases that modular forms admit a uniform description as $h$-tuples of $\Gamma_i$-invariant maps.  For this reason, a special role in our treatment will be played by maximal open compact subgroups.

\subsection*{Automorphic representations}

As it forms one of the core motivations of our work, we conclude this section by briefly describing the relationship between the spaces $M(W)$ of modular forms and automorphic representations of $\sfG$. Suppose that $W$ is defined over $F$ (cf.\ Gross \cite[\S 3]{Gross1}).  Since $G_\infty$ is compact, by averaging there exists a symmetric, positive-definite, $G_\infty$-invariant bilinear form
\[
\langle\,,\rangle:W_\infty\times W_\infty\lra F_\infty.
\]
where $W_\infty=W\otimes_FF_\infty$.  
%\todo{This used to involve the choice of a real place, but I don't see why you shouldn't want to do all at once.}
%You're right.  This is much better.
Then we have a linear map 
\[ \Psi:M(W) \lra \Hom_{G_\infty}(W_\infty, L^2(G\backslash(\Ghat \times G_\infty),F_\infty)) \]
by
\[ \Psi(f)(v)(\ghat,g_\infty) = \langle \rho(g_\infty)v,f(\ghat)\rangle \]
for $f \in M(W)$, $v \in W$, and $(\ghat,g_\infty) \in \Ghat \times G_\infty$.  The Hecke algebra $\cH(\Khat)$ acts on the representation space 
\[ \Hom_{G_\infty}(W_\infty, L^2(G\backslash (\Ghat \times G_\infty),F_\infty)) \] 
via its standard action on $L^2$ by convolution.  

Now, for a nonzero $v\in W_\infty$, define
\begin{align*}
\Psi_v : M(W) &\to L^2(G\backslash (\Ghat \times G_\infty), F_\infty) \\
f &\mapsto \Psi(f)(v).
\end{align*}
(In practice, it is often convenient to take $v$ to be a highest weight vector.)

\begin{proposition}
The map $\Psi_v$ is $\cH(\Khat)$-equivariant and induces a bijection between irreducible $\cH(\Khat)$-submodules of $M(W,\Khat)$ and automorphic representations $\pi$ of $\sfG(\bdA_F)$ such that 
\begin{enumroman}
\item $\pi(\Khat)$ has a nonzero fixed vector, and
\item $\pi_\infty$ is isomorphic to $\rho_\infty$.
\end{enumroman}
\end{proposition}

% \todo{You're getting this from Section 8 of Gross, right?  Cite needed.}

In particular, an $\cH(\Khat)$ eigenvector $f\in M(W,\Khat)$ gives rise to an automorphic representation.  Since automorphic representations are of such fundamental importance, explicit methods to decompose $M(W,\Khat)$ into its Hecke eigenspaces are of significant interest.  % \todo{I moved the algorithmic stuff to Section 5.  I think it will be clearer if we keep the two ``separate'', and to state the general method then explain our specialization, one right after the other.}

\section{Hermitian forms, classical groups, and lattices}

Having set up the general theory in the previous section, we now specialize to the case of orthogonal and unitary groups.  In this section, we introduce these classical groups; basic references are Borel \cite{Borel} and Humphreys \cite{Humphreys}.  

% \todo{We want Hermitian forms, not general sesquilinear forms---this imposes the condition (iii) below.  Right?}

\subsection*{Classical groups}

Let $F$ be a field with $\opchar F \neq 2$ and let $L$ be a commutative \'etale $F$-algebra equipped with an involution $\overline{\phantom{x}}:L\to L$ such that $F$ is the fixed field of $L$ under $\overline{\phantom{x}}$.  
% \todo{Using the notation like complex conjugation gives the right intuition on how the involution should be viewed; it is also the canonical standard involution on $L$ as an $F$-algebra.} 
Then there are exactly three possibilities for $L$:
\begin{enumerate}
\item $L=F$ and $\overline{\phantom{x}}$ is the identity;
\item $L$ is a quadratic field % \todo{it does not need to be unramified; this is not required by the \'etale condition} 
extension of $F$ and $\overline{\phantom{x}}$ is the nontrivial element of $\Gal(L/F)$; or
\item[$2'$.] $L \cong F\times F$ and $\overline{(b,a)}=(a,b)$ for all $(a,b) \in F \times F$.
\end{enumerate}
(As \'etale algebras, cases $2$ and $2'$) look the same, but we will have recourse to single out the split case.)

Let $V$ be a finite-dimensional vector space over $L$.  
% \todo{This is why I used $W$ for the weight, so we could reserve $V$ here.}  
Let
\[
\phi : V\times V\lra L
\]
be a \emph{Hermitian form} relative to $L/F$, so that:
\begin{enumroman}
\item $\phi( x+y, z) = \phi(  x,z) + \phi( y,z)$ for all $x,y,z \in V$;
\item $\phi( ax,y)=a\phi( x,y)$ for all $x,y \in V$ and $a \in L$; and
\item $\phi(y,x)=\overline{\phi(x,y)}$ for all $x,y \in V$.  % \todo{This was missing but needs to be there.}
\end{enumroman}
Further suppose that $\phi$ is \emph{nondegenerate}, so $\phi(x,V)=\{0\}$ for $x \in V$ implies $x=0$.  For example, the \emph{standard} nondegenerate Hermitian form on $V=L^n$ is
\begin{equation} \label{stdform}
\phi(x, y) = \sum_{i=1}^n x_i \overline{y_i}.
\end{equation}
% \todo{This doesn't follow all conventions; in the noncommutative case, the order matters (whether you conjugate on the left or right).  This doesn't really matter here, but maybe it already matters in the symplectic case.}

Let $\sfA$ be the (linear) algebraic group of automorphisms of $(V,\phi)$ over $F$: that is to say, for a commutative $F$-algebra $D$, we have
\[ \sfA(D)=\Aut_{D \otimes_F L}(V_F \otimes_F D, \phi). \]
(Note the tensor product is over $F$, so in particular we consider $V$ as an $F$-vector space and write $V_F$.)  More explicitly, we have
\[
\sfA(F) = \Aut_L(V, \phi) = \{T\in \GL(V) : \phi( Tx,Ty )=\phi(x,y)\}.
\]
Since $\phi$ is nondegenerate, for every linear map $T:V\to V$, there is a unique linear map $T^*:V\to V$ such that
\[
\phi(Tx,y)=\phi(x,T^*y)
\]
for all $x,y \in V$.  It follows that
\[
\sfA(F) = \{T\in \GL(V) : TT^*=1\}
\]
where the ${}^*$ depends on $\phi$.

The group $\sfA$ is reductive but is not connected in case $1$.  Let $\sfG$ be the connected component of the identity in $\sfA$.

In each of the three cases, we have the following description of $\sfG \leq \sfA$.

% \todo{We could also take the derived (commutator) subgroup of the orthogonal group.  This is related to the usual business about spinor genus versus genus, and in general these things are different even though very often they coincide.  The special orthogonal genus is the same as the orthogonal genus, for example.  We should decide which of these we actually want to compute on.  The best answer is all of them!  The real issues show up when considering the building, for which there are some important hypotheses.}

\begin{enumerate}
\item If $L=F$, then $\phi$ is a symmetric bilinear form over $F$ and $\sfG=\mathsf{SO}(\phi) \leq \mathsf{O}(\phi)=\sfA$ are the special orthogonal and orthogonal group of the form $\phi$.

\item If $L$ is a quadratic field extension of $F$, then $\phi$ is a Hermitian form with respect to $L/F$ and $\sfG=\mathsf{U}(\phi)=\sfA$ is the unitary group associated to $\phi$.    

\item[$2'$.] If $L=F\times F$, then actually we obtain a general linear group.  Indeed, let $e_1=(1,0)$ and $e_2=(0,1)$ be an $F$-basis of idempotents of $L$.  Then $V_1=e_1 V$ and $V_2=e_2 V$ are vector spaces over $F$, and the map $T \mapsto T|_{V_1}$ gives an isomorphism of $\sfG=\sfA$ onto $\mathsf{GL}(V_1)$.  
\end{enumerate}

\begin{remark}
It would also be profitable to consider other groups of symmetries of $(V,\phi)$, for example, the spin group $\textsf{Spin}(\phi)$ in case 1 and the special unitary group $\mathsf{SU}(\phi)$ in case 2.  We have simply made one such choice for the purposes of this article.
\end{remark}

\begin{remark}
To obtain symplectic or skew-Hermitian forms, we would work instead with signed Hermitian forms above.
\end{remark}

\begin{remark}
We have phrased the above in terms of Hermitian forms, but one could instead work with their associated quadratic forms $Q:V \to L$ defined by $Q(v)=\phi(v,v)$.  In characteristic $2$, working with quadratic forms has some advantages, but in any case we will be working in situations where the two perspectives are equivalent.
\end{remark}

\subsection*{Integral structure}\label{SS:global}

Suppose now that $F$ is a number field with ring of integers $\Z_F$.  By a \emph{prime} of $F$ we mean a nonzero prime ideal of $\Z_F$.

Let $(V,\phi)$ and $\sfG \leq \sfA$ be as above.  Since our goal is the calculation of algebraic modular forms, we insist that $G_\infty=\sfG(F_\infty)=\sfG(F \otimes_{\Q} \R)$ be compact, which rules out the case $2'$ (that $L=F\times F$) and requires that $F$ be totally real.  
% Thus
% \begin{enumerate}
% \item $L=F$ and $\sfG$ is the orthogonal group of the positive definite quadratic space $(V,\phi)$, or 
% \item $L/F$ is a quadratic field extension and $\sfG$ is the unitary group of the definite Hermitian space $(V,\phi)$.
% \end{enumerate}
% \todo{We should fix this up and use $\sfG_1$ in case (1), but it has to wait on resolving the matters in the previous section.}  

Let $\Z_L$ be the ring of integers of $L$.  Let $\Lambda \subset V$ be a \emph{lattice} in $V$, a projective $\Z_L$-module with rank equal to the dimension of $V$.  Suppose further that $\Lambda$ is \emph{integral}, so $\phi(\Lambda,\Lambda)\subseteq \Z_L$. 
Define the \emph{dual lattice} by
\[ \Lambda^{\#} = \{ x \in V : \phi(\Lambda, x) \subseteq \Z_L\}. \]
We say $\Lambda$ is \emph{unimodular} if $\Lambda^{\#}=\Lambda$.  

To a lattice $\Pi \subseteq V$ we associate the the lattice 
\[ \widehat{\Pi} = \Pi \otimes_{\Z_L} \widehat{\Z}_L \subset  \widehat{V} = V \otimes_{L} \widehat{L} \]
with $\Pi_\fp = \Pi \otimes_{Z_L} \Z_{L,\fp}$; we have $\Pi \otimes_{\Z_L} \Z_{L,\fp} = \Lambda \otimes_{\Z_L} \Z_{L,\fp}$ for all but finitely primes $\fp$.  Conversely, given a lattice $(\Pi_\fp)_\fp \subseteq \widehat{V}$ with $\Pi_\fp=\Lambda \otimes_{\Z_L} \Z_{L,\fp}$ for all but finitely many $\fp$, we obtain a lattice 
\[ \Pi = \{ x \in V : x \in \Pi_\fp \text{ for all $\fp$}\}. \]
(In fact, one can take this intersection over all localizations in $V$, not completions, but we do not want to confuse notation.)  These associations are mutually inverse to one another (\emph{weak approximation}), so we write $\widehat{\Pi} = (\Pi_\fp)_{\fp}$ unambiguously.

Let 
\begin{equation} \label{stabKhat}
\Khat = \{\ghat \in \Ghat : \ghat \widehat{\Lambda} = \widehat{\Lambda}\}
\end{equation}
be the stabilizer of $\widehat{\Lambda}$ in $\Ghat$.  Then $\Khat$ is an open compact subgroup of $\Ghat$.  Further, let 
\[ \Gamma = \{g \in G : g \Lambda = \Lambda\} \]
be the stabilizer of $\Lambda$ in $G$.  Then the group $\Gamma$ is finite, since it is a discrete subgroup of the compact group $G_\infty$ \cite[Proposition 1.4]{Gross1}.  % \todo{The alternative here would be to just define the group $\Gamma$ as it stands, without an attached group scheme!  What does the integral group scheme get us?}  
% Let $\fp$ be a nonzero prime ideal of $\Z_F$, and let $\Z_{F,\fp}$ and $\Z_{L,\fp}$ denote the completion of $\Z_F$ and $\Z_L$ at $\fp$, respectively.  
% Then for all primes $\fp \nmid \frakd(\Lambda)$, the lattice $\Lambda \otimes_{Z_L} \Z_{L,\fp}$ is (even and) unimodular.

\begin{remark}
In fact, by work of Gan and Yu \cite[Proposition 3.7]{GY}, there is a unique smooth linear algebraic group $\underline{\sfA}$ over $\Z_F$ with generic fiber $\sfA$ such that for any commutative $\Z_F$-algebra $D$ we have
\[ \underline{\sfA}(D) = \Aut_{D \otimes_{\Z_F} \Z_L}(\Lambda \otimes_{\Z_L} D, \phi).  \] 
As we will not make use of this, we do not pursue integral models of $\sfA$ any further here.
\end{remark}

We now consider the extent to which a lattice is determined by all of its localizations in this way: this extent is measured by the genus, which is in turn is given by a double coset as in Section 2, as follows.

\begin{definition}  
Let $\Lambda$ and $\Pi$ be lattices in $V$.  We say $\Lambda$ and $\Pi$ are ($G$-)\emph{equivalent} (or \emph{isometric}) if there exists $\gamma\in G$ such that $\gamma \Lambda=\Pi$.  
We say $\Lambda$ and $\Pi$ are \emph{locally equivalent} (or \emph{locally isometric}) if there exists $\ghat \in \Ghat$ such that
$\ghat \widehat{\Lambda} = \widehat{\Pi}$.  The set of all lattices locally equivalent to $\Lambda$ is called the \emph{genus} of $\Lambda$ and is denoted $\gen(\Lambda)$.
\end{definition}

For any $\ghat=(g_\fp)_\fp \in \Ghat$, we have
\[ \ghat \Lambdahat = \prod_{\fp} g_\fp \Lambda_\fp; \]
since $g_\fp \Lambda_\fp = \Lambda_\fp$ for all but finitely many $\fp$, by weak approximation, there is a unique lattice $\Pi \subseteq V$ such that $\widehat{\Pi} = \ghat \Lambdahat$.  By definition, $\Pi\in\gen(\Lambda)$ and every lattice $\Pi \in \gen(\Lambda)$ arises in this way.
Thus, the rule 
\[
(\ghat,\Lambda)\mapsto \Pi = \ghat \Lambda 
\] 
gives an action of $\Ghat$ on $\gen(\Lambda)$.  The stabilizer of $\Lambda$ under this action is by definition $\Khat$, therefore the mapping
\begin{align*}
\Ghat/\Khat &\to \gen(\Lambda) \\
\ghat \Khat &\mapsto \ghat \Lambda
\end{align*}
is a bijection of $G$-sets.  The set $G\backslash \gen(\Lambda)$ of isometry classes of lattices in $\gen(\Lambda)$ is therefore in bijection with the double-coset space (\emph{class set})
\[
Y=G\backslash \Ghat/\Khat.
\]
By Proposition \ref{Yisfinite} (or a direct argument, e.g.\ Iyanaga \cite[6.4]{Iyanaga1} for the Hermitian case), the genus $\gen(\Lambda)$ is the union of finitely many equivalence classes called the \emph{class number} of $\Lambda$, denoted $h=h(\Lambda)$.

In this way, we have shown that an algebraic modular form $f\in M(W,\Khat)$ can be viewed as a function on $\gen(\Lambda)$.  Translating the results of Section 1 in this context, if $\Lambda_1,\ldots,\Lambda_h$ are representatives for the equivalence classes in $\gen(\Lambda)$, then a map $f:\gen(\Lambda) \to W$ is determined by the finite set $f(\Lambda_1),\ldots,f(\Lambda_h)$ of elements of $W$.  % In practice, we will work with elements of the coset space $G(\hat{F})/K$ via the associated lattices.  
The problem of enumerating this system of representatives for $Y$ becomes the problem of enumerating representatives for the equivalence classes in $\gen(\Lambda)$, a problem which we will turn to in Section 5 after some preliminary discussion of elementary divisors in Section 4.

\section{Elementary divisors} \label{sec:Hecke}

% \todo{I think the decompositions that you wrote down in the examples are correct, but the decomposition of roots as you wrote it requires a maximal split torus.  So we need to dig into the Bruhat-Tits theory, or maneuver around it with the right definitions or standard references.}

In this section, we give the basic setup between elementary divisors and Hecke operators, providing a link to the neighbor method in the lattice setting.  The results are standard.  We work in the local case.  

Let $F$ be a local field of mixed characteristic with ring of integers $\Z_F$, let $\phi$ be a Hermitian form on $V$ relative to $L/F$, and let $\Z_L$ be the integral closure of $\Z_F$ in $L$ with uniformizer $P$. 

Suppose that $G$ is split, and let $T_s \subset G$ be a maximal split torus in $G$.  Let $W_s$ be the Weyl group of $(G_s,T_s)$.  Let
\[
X_*(T_s)=\Hom(T_s,\mathbb{G}_m)
\quad\text{and}\quad
X^*(T_s)=\Hom(\mathbb{G}_m,T_s)
\]
be the groups of characters and cocharacters of $T_s$, respectively.

\begin{theorem}
The Cartan decomposition holds:
\[
G = \bigsqcup_{\lambda\in X^*(T_s)/W_s} K \lambda(P) K.
\]
\end{theorem}

There is a standard method for producing a fundamental domain for the acton of $W_s$ on $X_*(T_s)$, allowing for a more explicit statement of the Cartan decomposition.  Let $\Phi^+\subseteq X^*(T_s)$ be a set of positive roots and let
\[
Y_s^+ = \{\lambda \in X_*(T_s) : \text{$\lambda(\alpha)\geq 0$ for all $\alpha\in\Phi^+$}\}.
\]

\begin{proposition} \label{Yplus}
$Y^+$ is a fundamental domain for the action of $W_s$ on $X_*(T_s)$.  Therefore, we have
\[
G =\bigsqcup_{\lambda\in Y^+}K \lambda(P) K.
\]
\end{proposition}

We now proceed to analyze this decomposition in Proposition \ref{Yplus} explicitly in our situation.  We suppose that $\Lambda$ is unimodular (in our methods, we will consider the completions at primes not dividing the discriminant), and we consider two cases.

We first consider the split case ($2'$) with $L\cong F \times F$.  Let 
\[
V_1=e_1V \cong F^n \quad\text{and}\quad \Lambda_1=e_1\Lambda \cong \Z_F^n.
\]
Recall that $T\mapsto T|_{V_1}$ identifies $G=\sfG(F)$ with $\GL(V_1)$.  Already having described $V_1$ and $\Lambda_1$ in terms of coordinates we have 
\[
G=\GL_d(F)\quad\text{and}\quad K=\GL_d(\Z_F).
\] 
Let $T \leq G$ be the subgroup of diagonal matrices, consisting of elements $t=\diag(t_1,\dots,t_n)$.  For $i=1,\dots,n$, define the cocharacter $\lambda_i:F^\times\to T$ by
\[
\lambda_i(a)=\diag(1,\ldots,1,a,1,\ldots,1),
\]
where $a$ occurs in the $i$-th component.  Then
\[
Y_s^+=\{\lambda_1^{r_1}\cdots\lambda_n^{r_n} : r_1\leq \cdots \leq r_n\}.
\]

\begin{proposition}[Elementary divisors; split case] \label{elemdiv-split}
Let $\Lambda$ and $\Pi$ be unimodular lattices in $V$ with $L \cong F \times F$.  Then there is a basis 
\[
e_1,\ldots,e_n
\]
of $\Lambda$ and integers 
\[
r_1\leq \cdots\leq r_n
\]
such that
\[
(\overline{P}/P)^{r_1}e_1,\ldots,(\overline{P}/P)^{r_n}e_n
\]
is a basis of $\Pi$.  Moreover, the sequence $r_1\leq \cdots \leq r_n$ is uniquely determined by $\Lambda$ and $\Pi$.
\end{proposition}

Now we consider the more difficult cases (1) and (2), which we can consider uniformly.  We have that either $L=F$ or the maximal ideal of $\Z_F$ is inert or ramified in $\Z_L$.  Let $\nu=\nu(\phi)$ be the \emph{Witt index} of $(V,\phi)$, the dimension of a maximal isotropic subspace.  Then $\nu=\dim T_s \leq n/2$ and $V$ admits a basis of the form
\[
e_1,\ldots,e_{\nu},g_1,\ldots,g_{n-2\nu}, f_1,\ldots,f_{\nu}
\]
such that
\[
\phi(e_i,e_j)=\phi(f_i,f_j)=0\quad\text{and}\quad \phi(e_i,f_j)=\delta_{ij}.
\]
% \todo{Does this even work when the residue field has characteristic $2$?}
In this basis, the matrix $\phi$ %\todo{the notation $\Phi$ conflicts with $\Phi^+$ as positive roots} is
is
\[
A(\phi)=\begin{pmatrix}
& & I\\
&\big(\phi(g_i,g_j)\big)_{i,j} & \\
I & &
\end{pmatrix}
\]
where $I$ is the $\nu \times \nu$ identity matrix.  The set of matrices of the form
\[
\begin{pmatrix}
\diag(t_1,\ldots,t_{\nu}) \\ & I \\ & & \diag(\overline{t}_1,\ldots,\overline{t}_{\nu})^{-1}
\end{pmatrix}
\]
constitute a maximal split torus in $G$.  Considering $\lambda_i$ as a cocharacter of $\GL_{\nu}(F)$ as above, define
\begin{align*}
\mu_i:F^\times &\to T  \\
a &\mapsto \mu_i(a)=\begin{pmatrix}\lambda_i(a)\\ &I\\ & & \lambda_i(\overline{a})^{-1}\end{pmatrix}.
\end{align*}
With these choices, we have
\[
Y_s^+=\{\mu_1^{r_1}\cdots\mu_{\nu}^{r_n} : r_1\leq \cdots\leq r_\nu\}.
\]

\begin{proposition}[Elementary divisors; nonsplit case] \label{elemdiv-nonsplit}
Let $\Lambda$ and $\Pi$ be unimodular lattices in $V$ with $L \not\cong F \times F$.  Then there is a basis 
\[
e_1,\ldots,e_{\nu},g_1,\ldots,g_{n-2\nu},f_1,\ldots,f_{\nu}
\]
of $\Lambda$ and integers 
\[
r_1\leq \cdots\leq r_{\nu}
\]
such that
\[
\overline{P}^{r_1}e_1,\ldots,\overline{P}^{r_\nu}e_\nu,g_1,\ldots,g_j,P^{-r_1}f_1,\ldots,P^{-r_\nu}f_\nu
\]
is a basis of $\Pi$.  Moreover, the sequence $r_1\leq \cdots \leq r_\nu$ is uniquely determined by $\Lambda$ and $\Pi$.
\end{proposition}

\section{Neighbors, lattice enumeration, and Hecke operators}

In this section, we describe the enumeration of representatives for equivalence classes in the genus of a Hermitian lattice.  We develop the theory of neighbors with an eye to computing Hecke operators in the next section.

The original idea of neighbors is due to Kneser \cite{Kneser}, who wished to enumerate the genus of a (positive definite) quadratic form over $\Z$.   Schulze-Pillot \cite{Schulze-Pillot} implemented Kneser's method as an algorithm to compute the genus of ternary and quaternary quadratic forms over $\Z$, and Scharlau and Hemkemeier \cite{SH} developed effective methods to push this into higher rank.  For Hermitian forms, Iyanaga \cite{Iyanaga2} used Kneser's method to compute the class numbers of unimodular positive definite Hermitian forms over $\Z[i]$ of dimensions $\leq 7$; later Hoffmann \cite{Hoffmann} pursued the method more systematically with results for imaginary quadratic fields of discriminants $d=-3$ to $d=-20$ and Schiemann \cite{Schiemann} extended these computations further for imaginary quadratic fields (as far as $d=-455$).  For further reference on lattices, see also O'Meara \cite[Chapter VIII]{OMeara}, Knus \cite{Knus}, Shimura \cite{Shimura-unitary}, and Scharlau \cite{Scharlau}.  

\subsection*{Neighbors and invariant factors}

Let $F$ be a number field with ring of integers $\Z_F$.  Let $\Z_L$ be a field containing $F$ with $[L:F] \leq 2$, ring of integers $\Z_L$, and involution $\overline{\phantom{x}}$ with fixed field $F$.  In particular, we allow the case $L=F$ and $\Z_L=\Z_F$.  Let $\phi$ be a Hermitian form on $V$ relative to $L/F$, and let $\Lambda \subset V$ be an integral lattice.  

If $\Pi \subset V$ is another lattice, then there exists a basis $e_1,\dots,e_n$ for $V$ and fractional ideals $\fA_1,\dots,\fA_n$ and $\fB_1,\dots,\fB_n$ of $\Z_L$ such that
\[ \Lambda=\fA_1 e_1 \oplus \dots \oplus \fA_n e_n \]
and
\[ \Pi = \fB_1 e_1 \oplus \dots \oplus \fB_n e_n \]
(a direct sum, not necessarily an orthogonally direct sum) satisfying
\[ \fB_1/\fA_1 \supseteq \dots \supseteq \fB_n/\fA_n. \]
The sequence $\fB_1/\fA_1, \dots, \fB_n/\fA_n$ is uniquely determined and called the \emph{invariant factors} of $\Pi$ relative to $\Lambda$.  Note that $\Pi \subseteq \Lambda$ if and only if the invariant factors are integral ideals of $\Z_L$.

Define the fractional ideal
\[
\fd(\Lambda,\Pi) = \prod_{i=1}^n \fB_i/\fA_i
\]
and let $\fd(\Lambda)=\fd(\Lambda^{\#},\Lambda)$, where $\Lambda^{\#} \supseteq \Lambda$ is the dual lattice of $\Lambda$.  Then in fact $\overline{\fd(\Lambda)}=\fd(\Lambda)$, so $\fd(\Lambda)$ arises from an ideal over $\Z_F$, which we also denote $\fd(\Lambda)$ and call the \emph{discriminant} of $\Lambda$.  In particular, $\Lambda$ is unimodular if and only if $\fd(\Lambda)=\Z_F$, and more generaly $\Lambda_\fp = \Lambda \otimes_{\Z_F} \Z_{F,\fp}$ is unimodular whenever $\fp$ is a prime of $F$ with $\fp \nmid \fd(\Lambda)$.

\begin{definition}
Let $\fP$ be a prime of $L$ and let $k \in \Z$ with $0 \leq k \leq n$.  An integral lattice $\Pi \subset V$ is a $\fP^k$-neighbor of $\Lambda$ if $\Pi$ has $k$ invariant factors $\overline{\fP}$ and $\fP^{-1}$, i.e., invariant factors
\[ \underbrace{\overline{\fP}, \dots, \overline{\fP}}_k,S,\dots,S,\underbrace{\fP^{-1},\dots,\fP^{-1}}_k \]
if $k\leq n/2$ and
\[ \underbrace{\overline{\fP}, \dots, \overline{\fP}}_{n-k},\underbrace{\overline{\fP}\fP^{-1},\dots,\overline{\fP}\fP^{-1}}_{2k-n},\underbrace{\fP^{-1},\dots,\fP^{-1}}_{n-k} \]
if $k > n/2$ and $\fP \neq \overline{\fP}$.
\end{definition}

We require that $\fP \neq \overline{\fP}$ if $k>n/2$ because the maximal isotropic subspaces in this case have dimension $\leq n/2$, by Proposition \ref{elemdiv-nonsplit}.   It follows from a comparison of invariant factors that $\Pi$ is a $\fP^k$-neighbor of $\Lambda$ if and only if
\begin{center} 
$\Pi/(\Lambda\cap \Pi) \cong  (\Z_L/\fP)^k$ and $\Lambda/(\Lambda\cap \Pi) \cong  (\Z_L/\overline{\fP})^k$.
\end{center}
A $\fP$-neighbor $\Pi$ of $\Lambda$ has the same discriminant $\fd(\Lambda)=\fd(\Pi)$.  

\begin{remark}
One may also define $N$-neighbors for $N$ a finitely generated torsion $\Z_L$-module.
\end{remark}

\subsection*{Neighbors and isotropic subspaces}

Let $\fP$ be prime of $L$ above $\fp$ and let $\fq=\fP \overline{\fP}$.  Then $\fq=\fp$ or $\fq=\fp^2$, where $\fp$ is the prime below $\fP$.  Suppose that $\fp \nmid \fd(\Lambda)$.  Let $X \subseteq \Lambda$ be a finitely generated $\Z_L$-submodule.  We say that $X$ is \emph{isotropic modulo $\fq$} if 
\[ \phi(x,y) \in \fq \text{ for all $x,y \in X$}. \]
Define the \emph{dual} of $X$ to be
\[
X^{\#} = \{y\in V : \phi(X,y)\subseteq \Z_L\}.
\]
Then 
\[
\overline{\fP} X^{\#} = \{y\in V : \phi(X,y)\subseteq \fP\}
\]
and $\Lambda \cap \overline{\fP} X^{\#} \subseteq \Lambda \subset V$ is a lattice. 

\begin{proposition} \label{isotropisneighbor}
Let $X \subseteq \Lambda$ be isotropic modulo $\fq$.  Then
\[ \Lambda(\fP,X) = \fP^{-1} X + (\Lambda \cap \overline{\fP} X^{\#}) \]
is a $\fP^k$-neighbor of $\Lambda$, where $k=\dim X/\fP X$.  Moreover, $\Lambda(\fP,X)=\Lambda(\fP,X')$ if and only if $X/\fP X = X'/\fP X' \subseteq \Lambda/\fP \Lambda$.
\end{proposition}

\begin{proof}
The integrality of $\Lambda(\fP,X)$ is easy to verify using the fact that $\phi(x,y)\in\fP \overline{\fP}$ for all $x,y \in X$ and $\phi(x,y) \in \fP$ for all $x \in X$ and $y \in \Lambda \cap \overline{\fP} X^{\#}$.  

First, we prove a claim: $\Lambda\cap \Lambda(\fP,X) = \Lambda \cap \overline{\fP} X^{\#}$.  The inclusion 
($\supseteq$) is clear.  For the reverse, suppose $y \in \Lambda \cap \Lambda(\fP,X)$, so $y=v+w$ with $v \in \fP^{-1}X$ and $w \in \Lambda \cap \overline{\fP} X^{\#}$; then $v=y-w \in \Lambda$ and 
\[ \phi(X,v) \subseteq \phi(X,\fP^{-1} X)=\phi(X,X)\overline{\fP}^{-1} \subseteq \fP\overline{\fP}\, \overline{\fP}^{-1} = \fP \]
so $v \in \Lambda \cap \overline{\fP} X^{\#}$ and thus $y=v+w \in  \Lambda \cap \overline{\fP} X^{\#}$ as well.  This proves the claim.

Now, choose a $\Z_L/\fP$-basis $x_1,\dots,x_k$ for $X/\fP X \subseteq \Lambda/\fP \Lambda$.  Consider the map 
\begin{align*}
\phi(\cdot,X)=\Lambda &\to (\Z_L/\overline{\fP})^k \\
y &\mapsto (\phi(y,x_i)) + \overline{\fP}.
\end{align*}
Since $\fp$ is coprime to $\fd(\Lambda)$, the Hermitian form $\phi$ is nondegenerate modulo $\overline{\fP}$; since $X$ is totally isotropic, it follows that $\phi(\cdot,X)$ is surjective.  Since $\phi(y,x)=\overline{\phi(x,y)}$ for all $x,y \in V$, we have that $\ker \phi(\cdot, X) = \Lambda \cap \overline{\fP} X^{\#}$.
Therefore, by the claim, we have 
\begin{equation} \label{oneincl}
\Lambda/(\Lambda \cap \Lambda(\fP,X))=\Lambda/(\Lambda \cap \overline{\fP} X^{\#}) \cong (\Z_L/\overline{\fP})^k.
\end{equation}
Next, we have 
\[ \fP^{-1}X \cap \Lambda = X. \]
Therefore,
\begin{equation} \label{LXcan}
\begin{aligned}
\Lambda(\fP,X)/(\Lambda\cap \Lambda(\fP,X)) &= (\fP^{-1}X + \Lambda \cap \overline{\fP} X^{\#})/(\Lambda \cap \overline{\fP} X^{\#}) \\
&\cong \fP^{-1} X/(\fP^{-1} X \cap (\Lambda \cap \overline{\fP} X^{\#})) \\
&= \fP^{-1} X/X \cong X/\fP X
\end{aligned}
\end{equation}
and $X/\fP X \cong (\Z_L/\fP)^k$.  Together with (\ref{oneincl}), we conclude that $\Lambda(\fP,X)$ is a $\fP^k$-neighbor.

For the final statement, if $X/\fP X = X'/\fP X'$ then it is clear that $\Lambda(\fP,X)=\Lambda(\fP,X')$ as this construction only depends on $X$ modulo $\fP$; the converse follows from the fact that the identification in (\ref{LXcan}) is canonical: if $\Lambda(\fP,X)=\Lambda(\fP,X')$ then $X/\fP X = X'/\fP X'$.
\end{proof}

% \todo{The latter statement about uniqueness contradicts Schiemann \cite[Lemma 3.1]{Schiemann}, but I don't see why.  It turns out that this only shows up in the inert and ramified cases, so we could hide and restrict to the split case.}

\begin{proposition} \label{neighborisisotrop}
Let $\Pi$ be a $\fP^k$-neighbor of $\Lambda$.  Suppose that $\fq=\fP\overline{\fP}$ is coprime to $\fd(\Lambda)$.  Then there exists $X \subseteq \Lambda$ isotropic modulo $\fP$ with $\dim X/\fP X = k$ such that $\Pi=\Lambda(\fP,X)$. 
\end{proposition}

\begin{proof}
Let $X$ be the $\Z_L$-submodule of $\fP \Pi$ generated by a set of representatives for $\fP \Pi$ modulo $\fP(\Pi \cap \Lambda)$.  Then $X$ is finitely generated and $\phi(X,X) \subseteq \fP\overline{\fP}=\fq$ by the integrality of $\Pi$.  Since $\Pi$ is a $\fP$-neighbor, we have 
\[ \Pi/(\Lambda\cap \Pi) \cong (\Z_L/\fP)^k \] 
so $\fP \Pi\subseteq \Lambda\cap \Pi\subseteq \Lambda$, showing that $X \subseteq \Lambda$ and $X/\fP X \cong (\Z_L/\fP)^k$ by nondegeneracy.

Next, we prove that $\Pi \subseteq \Lambda(\fP,X)$.  If $y\in \Pi$, then by the integrality of $\Pi$,
\[
\phi(X,y)\subseteq \phi(\fP \Pi,y)=\fP \phi(\Pi,y)\subseteq \fP.
\]
Therefore, $\Lambda\cap \Pi \subseteq \Lambda \cap \overline{\fP} X^{\#}$.  But 
\[
\Lambda\cap \Pi \subsetneq \fP^{-1}X+\Lambda\cap \Pi\subseteq \Pi
\]
and
\begin{align*}
(\fP^{-1} X + (\Lambda \cap \Pi))/(\Lambda \cap \Pi) &\cong \fP^{-1} X/(\fP^{-1} X \cap (\Lambda \cap \Pi)) \\
&\cong X/(X \cap \fP(\Lambda \cap \Pi)) \cong X/\fP X
\end{align*}
by construction; since $\Pi/(\Lambda\cap \Pi) \cong (\Z_L/\fP)^k$ and $X/\fP X \cong (\Z_L/\fP)^k$, we conclude $\fP^{-1}X+(\Lambda\cap \Pi)=\Pi$.  Thus
\[ \Pi =  \fP^{-1}X + (\Lambda \cap \Pi) \subseteq \fP^{-1} + (\Lambda \cap \overline{\fP} X^{\#}) = \Lambda(\fP,X) \]
as claimed.

But now since both $\Lambda(\fP,X)$ and $\Pi$ are $\fP^k$-neighbors of $\Lambda$, they have the same invariant factors relative to $\Lambda$, so the containment $\Pi \subseteq \Lambda(\fP,X)$ implies $\Pi=\Lambda(\fP,X)$.  
\end{proof}

Putting together Propositions \ref{isotropisneighbor} and \ref{neighborisisotrop}, we obtain the following corollary.

\begin{corollary}
The map $X \mapsto \Lambda(\fP,X)$ gives a bijection between isotropic subspaces $X \subseteq \Lambda$ modulo $\fP$ with $\dim X/\fP X=k$ and $\fP^k$-neighbors of $\Lambda$.
\end{corollary}

From this corollary, we see that by taking a flag inside an isotropic subspace $X$ with $\dim X/\fP X=k$, every $\fP^k$-neighbor $\Pi$ can be obtained as a sequence 
\[ \Lambda_1=\Lambda(\fP,X_1), \Lambda_2=\Lambda_1(\fP,X_2),\dots, \Pi=\Lambda_{k-1}(\fP,X_{k-1}) \] 
of $\fP$-neighbors.  However, not all such $k$-iterated neighbors are $\fP^k$-neighbors: $\Lambda$ is itself a $\fP$-neighbor of any of its $\fP$-neighbors, for example.  

\subsection*{Neighbors, the genus, and strong approximation}

The $\fP$-neighbors can also be understood very explicitly when $\fP$ is \emph{odd} (i.e., $\fP \nmid 2$).  

Let $X \subseteq \Lambda$ be isotropic modulo $\fP$ with $\dim X/\fP X =k$.  We revisit the elementary divisor theory of Section 4.  There is a $\Z_{L,\fp}$-basis $x_1,\dots,x_n$ for $\Lambda_\fp$ such that $x_1,\dots,x_k$ is a basis for $X_\fp$ such that a basis for $\Lambda(\fP,X)_\fp$ is 
\[ (\overline{P}/P)x_1,\dots, (\overline{P}/P)x_k,  x_{k+1}, \dots, x_n \]
if $\fP \neq \overline{\fP}$ and is
\[ Px_1,\dots, Px_k, x_{k+1}, \dots, x_{n-k}, P^{-1} x_{n-k+1}, \dots, P^{-1}x_{n},  \]
if $\fP=\overline{\fP}$, where $P$ is a uniformizer of $\fP$.  In the first case, since $q=\overline{P}P^{-1}$ has $q\overline{q}=1$, this diagonal change of basis is an isometry and thus $\Lambda(\fP,X)_\fp \cong \Lambda_\fp$; a direct calculation in the latter case shows again that it is an isometry.  

% \begin{rmk}
% If $\fP \mid 2$---a case of interest as there are few neighbors---then this argument breaks down since you can't split off hyperbolic planes.  Indeed, you have to make an assumption like $\phi(x,y) \in 2\fq$ instead of just $\phi(x,y) \in \fq$; I think then it still goes through, but we would have to check this.}

Since the invariant factors of a $\fP^k$-neighbor are supported over $\fp$, we have proven the following lemma.

\begin{lemma}
Let $\Pi$ be a $\fP$-neighbor of $\Lambda$ with $\fp$ below $\fP$ and $\fp \nmid \fd(\Lambda)$.  Suppose that $\fp$ is odd.  Then $\Pi$ belongs to the genus of $\Lambda$.
\end{lemma}

Now we form the graph of $\fP^k$-neighbors: the vertices consist of a set of equivalence classes of lattices in the genus of $\Lambda$, and for each vertex $\Pi$ we draw a directed edge to the equivalence class of each $\fP^k$-neighbor of $\Pi$.  This graph is $\kappa$-regular, where $\kappa$ is the number of isotropic subspaces of $\Lambda_\fp$ modulo $\fP$ of dimension $k$---since all lattices in the genus are isomorphic.  If, for example, $\phi$ is the standard form (so is totally split) and $\fP \neq \overline{\fP}$, then this number is simply the cardinality of the Grassmanian $\Gr(n,k)(\F_\fP)$ of subspaces of dimension $k$ in a space of dimension $n$, and we have the formula
\begin{equation}
\#\Gr(n,k)(\F_q)=\frac{(q^n-1)(q^n-q)\cdots (q^n-q^{k-1})}{(q^k-1)(q^k-q)\cdots (q^k-q^{k-1})}.
\end{equation}

% \todo{I'm not sure that this contradicts Schiemann \cite[Lemma 3.3]{Schiemann} in the inert and ramified case, but I worry!}

To conclude, we show that in fact the entire genus can be obtained via iterated $\fP$-neighbors; this is equivalent to the assertion that the graph of $\fP$-neighbors is connected.  

First, we need the following important result, a consequence of strong approximation.  For the orthogonal case $L=F$, see Eichler \cite{Eichlerbook}, Kneser \cite{Kneser}, or O'Meara \cite[\S 104]{OMeara}; for the unitary case $L/F$, see Shimura \cite[Theorem 5.24, 5.27]{Shimura-unitary} (and Schiemann \cite[Theorem 2.10]{Schiemann}); and for a further perspectives, see the survey by Kneser \cite{Kneser-strong}.

We say that a lattice $\Lambda$ is \emph{nice at the ramified primes} if for all $\fq$ ramified in $L/F$, the lattice $\Lambda_\fq$ splits a one-dimensional sublattice.  If $n$ is odd or $L=F$ or $\Lambda$ is even unimodular, this condition holds.  

Let $\Cl(\Z_L)$ be the class group of $\Z_L$ and let $\Cl(\Z_L)^{\langle \overline{\phantom{x}}\rangle}$ be the set of those classes that have a representative $\fA$ with $\fA=\overline{\fA}$.  Note $\Cl(\Z_L)=\Cl(\Z_L)^{\langle \overline{\phantom{x}}\rangle}$ if $L=F$.

\begin{theorem}[Strong approximation] \label{strongapprox}
Suppose that 
\begin{enumroman}
\item $L=F$, $n \geq 3$, and $\fd(\Lambda)$ is squarefree, or 
\item $[L:F]=2$, $n \geq 2$, and $\Lambda$ is nice at the ramified primes.
\end{enumroman}
Let $S$ be a nonempty set of primes of $L$ coprime to $\fd(\Lambda)$ that represent all elements in $\Cl(\Z_L)/\Cl(\Z_L)^{\langle \overline{\phantom{x}}\rangle}$.  

Then every lattice in $\gen(\Lambda)$ is equivalent to a lattice $\Pi$ with $\Pi_\fq = \Lambda_\fq$ for all primes $\fq$ below a prime $\fQ \not\in S$.
\end{theorem}

\begin{proof}[sketch]
The hypotheses $n \geq 3$ and $n \geq 2$ are necessary, as they imply that the corresponding spin or special unitary group is simply connected; they also further imply that for all primes $\fp \nmid \fd(\Lambda)$ below a prime $\fP \in S$, the form $\phi_\fp$ on $V_\fp$ is isotropic, so $G_\fp$ is not compact.  Since the set $S$ is nonempty, strong approximation then implies that every lattice in the spin genus or special genus of $\Lambda$ is equivalent to a lattice $\Pi$ as in the statement of the theorem.  Finally, the hypothesis that $\fd(\Lambda)$ is squarefree in the orthogonal case implies that the genus of $\Lambda$ contains only one spinor genus (see O'Meara \cite[\S 102]{OMeara} or Kneser \cite{Kn3}); and the difference between the special genus and the genus of $\Lambda$ is measured by the group $\Cl(\Z_L)/\Cl(\Z_L)^{\langle \overline{\phantom{x}}\rangle}$ by work of Shimura when $\Lambda$ is nice at the ramified primes.
\end{proof}

\begin{remark}
These are not the minimal set of hypotheses in which strong approximation holds, but they will suffice for our purposes; see the references above for a more comprehensive treatment.
\end{remark}

We then have the following corollary; see also Kneser \cite[\S 2]{Kneser}, Iyanaga \cite[2.8--2.11]{Iyanaga2}, and Hoffmann \cite[Theorem 4.7]{Hoffmann}.

\begin{corollary}  \label{getallgenus}
Under the hypotheses of Theorem \textup{\ref{strongapprox}}, every lattice in $\gen(\Lambda)$ can be obtained as a sequence of $\fP$-neighbors for $\fP \in S$.  
\end{corollary}

\begin{proof}
Let $\Pi \in \gen(\Lambda)$.  By strong approximation, we may assume that $\Pi_\fQ=\Lambda_\fQ$ for all $\fQ \not\in S$.  

First, suppose that $\Pi_\fQ = \Lambda_\fQ$ for all $\fQ \neq \fP$.  Then $\fP^m \Pi \subseteq \Lambda$ for some $m \in \Z_{\geq 0}$.  We proceed by induction on $m$.  If $m=0$, then $\Pi \subseteq \Lambda$ and since $\fd(\Pi)=\fd(\Lambda)$ we have $\Pi=\Lambda$.  

Suppose $m>0$, and choose $m$ minimal so that $\fP^m \Pi \subseteq \Lambda$.  Let $X$ be the $\Z_L$-submodule of $\fP^m \Pi$ generated by a set of representatives for $\fP^m \Pi$ modulo $\fP^m \Pi \cap \fP\Lambda$.  Then $X \subseteq \fP^m \Pi \subseteq \Lambda$.  Now consider again the proof of Proposition \ref{neighborisisotrop}.  We see that $X$ is isotropic (in fact, $\phi(X,X) \in \fq^m$).  Form the neighbor $\Lambda(\fP,X)$.  Then $\Lambda(\fP,X)$ can be obtained from a sequence of $\fP$-neighbors of $\Lambda$.  

Now we have
\[ \fP^{m-1} \Pi = \fP^{-1} X + (\Lambda \cap \fP^{m-1} \Pi) \subseteq \fP^{-1} X + (\Lambda \cap \overline{\fP} X^{\#}) = \Lambda(\fP,X) \]
since 
\[ \phi(X,\fP^{m-1} \Pi) \subseteq \phi(\fP^m \Pi,\fP^{m-1} \Pi) \subseteq \fP \fq^{m-1}\phi(\Pi,\Pi) \subseteq \fP.  \]
Therefore, by induction, $\fP^{m-1} \Pi$ can be obtained by a repeated $\fP$-neighbor of $\Lambda(\fP,X)$, and we are done by transitivity.

In the general case, we simply repeat this argument for each prime $\fP$ in $\Z_L$.
\end{proof}

\subsection*{Hecke operators}

We now connect the theory of neighbors to Hecke operators via elementary divisors and the Cartan decomposition as in the previous section.  Specifically, we compute the action of $\cH(\sfG(F_\fp),K_\fp)$ on $M(W,\Khat)$ for primes $\fp \nmid \fd(\Lambda)$.  In this case, the corresponding lattice is unimodular.

% \[
% \underline{G}_\fp = \underline{G}\times_{\Spec\fo}\Spec\fo_\fp
% \]
% is smooth.  

As should now be evident from this description in terms of maximal isotropic subspaces, the Hecke operator acts on a lattice by a summation over its neighbors.  We record this in the following theorem.

\begin{theorem} \label{Heckeisneighbor}
Let $\phat \in \Ghat$ correspond to the sequence 
\[ 0 \leq \dots \leq 0\leq \underbrace{1 \leq \dots \leq 1}_k \]
in Proposition \ref{elemdiv-nonsplit} or \ref{elemdiv-split}.  Write $\Khat \phat \Khat = \bigsqcup \phat_j \Khat$.  Then for any $\xhat \in \Ghat$, the set of lattices
\[ \Pi_j = \xhat \phat_j \Lambda = \xhat \phat_j \Lambdahat \cap V \]
is in bijection with the set of $\fP^k$-neighbors of $\xhat \Lambda$.
\end{theorem}

\begin{proof}
Each $\Pi_j$ is indeed a $\fP^k$-neighbor, as is visible by looking at the corresponding invariant factors.  Since the $\fP^k$-neighbors are in bijection with maximal isotropic subspaces and so are the cosets $\phat_j$, these sets are in bijection.
\end{proof}

\section{Algorithmic details}

Having discussed the theory in the previous sections, we now present our algorithm for using lattices to compute algebraic modular forms.  

\subsection*{General case}

We first give a general formulation for algebraic groups: this general blueprint can be followed in other situations (including symplectic groups, exceptional groups, etc.).  We compute the space $M(W,\Khat)$ of algebraic modular forms of weight $W$ and level $\Khat$ on a group $\sfG$.  To begin with, we must decide upon a way to represent in bits the group $\sfG$, the open compact subgroup $\Khat$, and the $G$-representation $W$ so we can work explicitly with these objects.  Then, to compute the space $M(W,\Khat)$ as a module for the Hecke operators, we carry out the following tasks:
\begin{enumalg}
\item Compute representatives $\xhat_i \Khat$ ($i=1,\dots,h$) for $G\backslash \Ghat / \Khat$, as in (\ref{GGK}), compute $\Gamma_i=G \cap \xhat_i \Khat \xhat_i^{-1}$, and initialize
\[ H = \bigoplus_{i=1}^h H^0(\Gamma_i, W). \]
Choose a basis of (characteristic) functions $f$ of $H$.  
\item Determine a set of Hecke operators $T(\phat)$ that generate $\cH(\Khat)$, as in Section \ref{sec:Hecke}.  For each such $T(\phat)$:
\begin{enumalgalph} 
\item Decompose the double coset $\Khat \phat \Khat$ into a union of right cosets $\phat_j \Khat$, as in (\ref{KpKdoubletoright});
\item For each $\xhat_i$ and $\phat_j$, find $\gamma_{ij} \in G$ and $j^*$ so that 
\[ \xhat_i \phat_j \Khat = \gamma_{ij} \xhat_{j^*} \Khat. \]
\item Return the matrix of $T(\phat)$ acting on $H$ via the formula
\[ (T(\phat) f)(\xhat_i) = \sum_j \gamma_{ij} f_m(\xhat_{j^*}) \]
for each $f$ in the basis of $H$.
\end{enumalgalph}
\end{enumalg}

In step 2c, since each function $f_m$ is a characteristic function, we are simply recording for each occurrence of $j^*$ an element of $G$.  

We now turn to each of the pieces of this general formulation in our case.

\subsection*{Representation in bits}

We follow the usual algorithmic conventions for number fields \cite{Cohen1}.  A Hermitian form $(V,\phi)$ for $L/F$ is represented by its Gram matrix.  We represent a $\Z_F$-lattice $\Lambda \subset V$ by a \emph{pseudobasis} over $\Z_F$, writing
\[ \Lambda = \fA_1 x_1 \oplus \dots \oplus \fA_n x_n \]
with $x_1,\dots,x_n \in V$ linearly independent elements and $\fA_i \subset L$ fractional $\Z_L$-ideals \cite{Cohen2}.  The open compact subgroup $\Khat$ is the stabilizer of $\Lambda$ by (\ref{stabKhat}) so no further specification is required.

The irreducible, finite dimensional representations of $G$ are given by highest weight representations.  The theory is explained e.g.\ by Fulton and Harris \cite{FultonHarris}, and in \textsf{Magma} there is a construction of these representations \cite{dG01,CMT04}, based on the \textsf{LiE} system \cite{vLCL92}.  

\subsection*{Step 1: Enumerating the set of representatives}

We enumerate a set of representatives $\xhat_i \Khat$ for $G \backslash \Ghat / \Khat$ using the results of Sections 4 and 5.  For this, we will use Corollary \ref{getallgenus}, and so we must assume the hypotheses of Theorem \ref{strongapprox}, namely:
\begin{enumroman}
\item $L=F$, $n \geq 3$, and $\fd(\Lambda)$ is squarefree, or 
\item $[L:F]=2$, $n \geq 2$, and $\Lambda$ is nice at the ramified primes.
\end{enumroman}

Next, according to Corollary \ref{getallgenus} we compute a nonempty set of primes $S$ of $L$ coprime to $2\fd(\Lambda)$ that represent all elements in $\Cl(\Z_L)/\Cl(\Z_L)^{\langle \overline{\phantom{x}} \rangle}$.  By the Chebotarev density theorem, we may assume that each prime $\fP$ is split in $L/F$ if $L \neq F$.  There are standard techniques for computing the class group due to Buchmann (see Cohen \cite[Algorithm 6.5.9]{Cohen2} for further detail).  We compute the action of the involution $\overline{\phantom{x}}$ on $\Cl(\Z_L)$ directly and then compute the subgroup $\Cl(\Z_L)^{\langle \overline{\phantom{x}} \rangle}$ fixed by $\overline{\phantom{x}}$ and the corresponding quotient using linear algebra over $\Z$.

Next, we traverse the graph of $\fP$-neighbors for each $\fP \in S$.  To do this, we perform the following tasks:
\begin{enumalgalph1}
\item Compute a basis for $\Lambda_\fP$ as in Propositions \ref{elemdiv-split} and \ref{elemdiv-nonsplit} according as $L=F$ or $L \neq F$.
\item Compute the one-dimensional isotropic subspaces modulo $\fP$ in terms of the basis $e_i$ for the maximal isotropic subspace.
\item For each such subspace $X$, compute the $\fP$-neighbor $\Lambda(\fP,X)=\fP^{-1} X + \overline{\fP}X^{\#}$ using linear algebra.  
\item Test each neighbor $\Lambda(\fP,X)$ for isometry against the list of lattices already computed.  For each new lattice $\Lambda'$, repeat and return to step a with $\Lambda'$ in place of $\Lambda$.  
\end{enumalgalph1}

Since the genus is finite, this algorithm will terminate after finitely many steps.  

\begin{remark}
One can also use the exact mass formula of Gan and Yu \cite{GY} and Gan, Hanke, and Yu \cite{GHY} as a stopping criterion, or instead as a way to verify the correctness of the output.
\end{remark}

Further comments on each of these steps.

First, in steps 1a--1b we compute a basis.  When $[L:F]=2$, this is carried out as in Section 5 via the splitting $L_\fp \cong F_\fp \times F_\fp$.  When $L=F$, we use standard methods including diagonalization of the quadratic form: see e.g.\ work of the second author \cite{VoightMaxOrder} and the references therein, including an algorithm for the normalized form of a quadratic form over a dyadic field, which at present we exclude.  From the diagonalization, we can read off the maximal isotropic subspace, and this can be computed by working not over the completion but over $\Z_F/\fp^e$ for a large $e$.  
Next, in step 1c we compute the neighbors.  This is linear algebra.  
%\todo{Or can't you just read this off of the computed basis?  Schiemann uses O'Meara's way of writing the %lattice with $x \in X$ as part of a basis.  For quaternion algebras, we just compute the (analogue of the) %dual as a separate fast algorithm.  How do you do it?}
Step 1d, isometry testing, is an important piece in its own right, which we discuss in the next subsection; as a consequence of this discussion, we will also compute $\Gamma_i=\Aut(\Lambda_i)$.  From this, the computation of a basis for $H=\bigoplus_{i=1}^h H^0(\Gamma_i, W)$ is straightforward.

\subsection*{Isometry testing}

To test for isometry, we rely on standard algorithms for quadratic $\QQ$-spaces and $\ZZ$-lattices even when computing relative to a totally real base field $F$ or a CM extension $L/F$.  Let $a_1,\ldots,a_d$ be a $\ZZ$-basis for $\Z_L$ with $a_1=1$, and let $x_1,\ldots,x_n$ be a basis of $V$.  Then 
\[
\{a_i x_j \}_{\substack{i=1,\dots,d \\ j=1,\dots,n}}
\]
is a $\QQ$-basis of $V$.  Define $\QQ$-bilinear pairings
\[
\phi_i:V\times V\lra \QQ\quad\text{by}\quad \phi_i(x,y)=\tr_{L/\QQ}\phi(a_i x,y).
\]
Since $a_1=1$ and $\phi$ is a definite Hermitian form on $V$ over $L$, $\phi_1$ is a positive definite, symmetric, bilinear form on $V$ over $\Q$.  In other words, $(V,\phi_1)$ is a quadratic $\QQ$-space.  The $L$-space $(V,\phi)$ can be explicitly recovered from $(V,\phi_1)$, together with the extra data $\phi_2,\ldots,\phi_d$ by linear algebra.  Note that the forms $\phi_2,\ldots,\phi_d$ are in general neither symmetric nor positive definite.

\begin{lemma}  \label{isomZZ}
Let $\Lambda$ and $\Pi$ be lattices in $V$.  A $\ZZ$-linear map $f:\Lambda\to \Pi$ is an $\Z_L$-linear isometry if and only if each $\phi_i$ is invariant under $f$.
\end{lemma}

Using Lemma \ref{isomZZ}, we reduce the problem of testing if two Hermitian lattices over $\Z_F$ are isometric to a problem of testing if two lattices over $\Z$ are isometric in a way which preserves each $\phi_i$.  For this, we rely on the algorithm of Plesken and Souvignier \cite{PS}, which matches up short vectors and uses other tricks to rule out isometry as early as possible, and has been implemented in \textsf{Magma} \cite{Magma} by Souvignier, with further refinements by Steel, Nebe, and others.

For each representative lattice found, we compute a set of invariants to quickly rule out isometry whenever possible.  These invariants include things like the sizes of the automorphism groups, the first few terms in the theta series, and invariants of sublattices (e.g.\ those generated by short vectors).

\begin{remark}
An essential speed up in the case of Brandt modules is given by Demb\'el\'e and Donnelly \cite{DD} (see also Kirschmer and the second author \cite[Algorithm 6.3]{KV}).  To decide if two right ideals $I,J$ in a quaternion order $\calO$ are isomorphic, one first considers the colon ideal $(I:J)_L=\{\alpha \in B : \alpha J \subseteq I\}$ to reduce the problem to show that a single right ideal is principal; then one scales the positive definite quadratic form over $\Q$ by an explicit factor to reduce the problem to a single shortest vector calculation.  It would be very interesting to find an analogue of this trick in this context as well.
\end{remark}

\subsection*{Step 2: Hecke operators}

Essentially all of the work to compute Hecke operators has already been set up in enumerating the genus in Step 1.  The determination of the Hecke operators follows from Sections 4 and 5, and their explicit realization is the same as in Step 1a.  We work with those Hecke operators supported at a single prime.  In Step 2a, from Theorem \ref{Heckeisneighbor}, the double coset decomposition is the same as set of $\fP$-neighbors, which we compute as in Step 1.  In Step 2b, we compute the isometry $\gamma_{ij}$ using isometry testing as in the previous subsection: we quickly rule out invalid candidates until the correct one is found, and find the corresponding isometry.  Finally, in Step 2c, we collect the results by explicit computations in the weight representation.

\section{Examples}
In this section, we illustrate our methods by presenting the results of some explicit computations for groups of the form $\sfG=\sfU_{L/F}(3)$, relative to a CM extension $L/F$, where $L$ has degree $2$, $4$, or $6$.  We compute the Hecke operator at unramified, degree one primes $\fp$ of $L$ corresponding to summing over $\fp$-neighbours of a lattice.

\begin{remark}
We made several checks to ensure the correctness of our programs.  First, we checked that matrices of Hecke operators for $\fp$ and $\fq$ with $\fp\neq \fq$ commuted.  (They did.)  Additionally, a known instance of Langlands functoriality implies that forms on $U(1)\times U(1)\times U(1)$ transfer to $U(3)$.  Checking that resulting endoscopic forms occur in the appropriate spaces~\cite[\S\S4.2, 4.6]{Loeffler} also provided a useful test of our implementation.  (It passed.)
\end{remark}

\begin{example}
$U_{L/F}(3)$, $L=\QQ(\sqrt{-7})$, $F=\QQ$, weights $(0,0,0)$ and $(3,3,0)$:

\begin{table}[htdp]
\caption{Computation of $T_{\fp,1}$ on $M(\QQ)$ for unramified, degree one $\fp\subset\ZZ_L$ with $2<N(\fp)<200$.}
\begin{center}
\begin{tabular}{|c|c|c|c|c|c|c|c|c|c|c|c}\label{tbl:quadratic1}
$N(\fp)$      & 2    & 11 & 23   & 29   & 37   & 43   & 53   & 67   & 71   & 79   & 107\\
\hline 	   
time (s) & 0.02 & 0.07 & 0.18 & 0.28 & 0.42 & 0.57 & 0.82 & 1.35 & 1.46 & 1.86 & 3.73 \\
\hline 
$a_p$      & 7  &  133  & 553  & 871  & 1407 & 1893 & 2863 & 4557 & 5113 & 6321 & 11557\\
\hline
$b_p$ & -1 & 5 & 41 & -25& -1   & 101  & 47   & -51  & 185  & -15  & 293  \\
\hline  \\ \\ \hline
     $N(\fp)$ & 109  & 113  & 127  & 137  & 149  & 151  & 163   & 179   & 191   & 193   & 197\\
     \hline
time (s) & 4.18 & 4.59 & 5.85 & 7.08 & 8.56 & 9.04 & 10.88 & 13.78 & 16.92 & 17.22 & 17.29\\
\hline
$a_p$ & 11991 & 12883 & 16257 & 18907 & 22351& 22953& 26733& 32221& 36673& 37443& 39007\\
\hline
$b_p$ & 215 & -109 & 129 & -37 & 335 & 425 & 237 & -163 & -127 & 131 & 479
\end{tabular}
\end{center}
\label{default}
\end{table}%

Here, we extend aspects of the calculation in the principal example of~\cite{Loeffler}.  In this case, the class number of the principal genus of rank $3$ Hermitian lattices for $L/\QQ$ is $2$, with classes represented by the standard lattice $\Lambda_1=\ZZ_L^3$ and the lattice $\Lambda_2\subset L^3$ with basis
\[
(1-\omega,0,0),\quad(1,1,0),\quad \tfrac{1}{2}(-3+\omega, -1+\omega, -1+\omega).\qquad
\left(\omega = \tfrac{1}{2}(1+\sqrt{-7})\right)
\]
The lattices $\Lambda_1$ and $\Lambda_2$ are $2$-neighbours:
\[
\omega\Lambda_2\subset\Lambda_1,\quad \bar{\omega}\Lambda_1\subset\Lambda_2.
\]
(Representatives for the ideal classes in the principal genus were computed in the first place by constructing the $2$-neighbour graph $\Lambda_1$.)  It follows that the space of $M(\QQ)$ of algebraic modular forms for $U_{L/F}(3)$ with trivial coefficients is simply the $2$-dimensional space of $\QQ$-valued function on 
$\{[\Lambda_1],[\Lambda_2]\}$.  We obtain two distinct systems $a_\fp$ and $b_\fp$ of Hecke eigenvalues occurring in $M(\QQ)$.  (See Table~\ref{tbl:quadratic1}.)  We point out observations of Loeffler considering the nature of the corresponding algebraic modular forms:  First, observe that the system $a_\fp$ is ``Eisenstein'', in the sense that the eigenvalues of $T_{\fp,1}$ is the degree of the Hecke operator $T_{\fp,1}$:
\[
a_\fp = N(\fp)^2+N(\fp)+1.
\]
Equivalently, the corresponding algebraic modular form is the lift from $U(1)\times U(1)\times U(1)$ of $\chi_{\text{triv}}\times \chi_{\text{triv}}\times \chi_{\text{triv}}$.  The algebraic modular form with system of eigenvalues $b_\fp$ is also a lift form $U(1)\times U(1)\times U(1)$:
\[
b_\fp = \fp^2 + \fp\bar{\fp} + \bar{\fp}^2.
\]
Here, we abuse notation and write $\fp$ for either of its generators, and note that the expression on the right in the above is independent of this choice.

We now consider analogous computations involving forms of higher weight.  The space $M(V_{3,3,0})$ the associated to the above data has dimension $4$, while the representation space $V_{3,3,0}$ itself has dimension 64.  Loeffler~\cite{Loeffler} showed that $M(V_{3,3,0})$ splits as the direct sum of two $2$-dimensional, Hecke-stable subspaces not diagonalizable over $\QQ(\sqrt{-7})$:
\[
M(V_{3,3,0}) = W_1\oplus W_2.
\]
One of these spaces arises as the lift involving a $2$-dimensional Galois conjugacy class of classical eigenforms in $S_9(\Gamma_1(7))$ via a lifting from the endoscopic subgroup $U(1)\times U(2)$ of $U(3)$.  The other corresponds to a Galois conjugacy class of nonendoscopic forms, whose associated $\ell$-adic Galois representations $\rho:G_L\to GL_2(\QQ_\ell)$ are irreducible.

We consider the corresponding modular space $M(V_{3,3,0}/\FF_7)$ of mod $(\sqrt{-7})$-modular forms of weight $(3,3,0)$.  We computed the Hecke operators $T_{\fp,1}$ for unramified, degree one primes $\fp\subset\ZZ_L$ of norm at most $100$.  In Table~\ref{tbl:quadratic2}, we present the corresponding run-times.

\begin{table}[htdp]
\caption{Computation of $T_{\fp,1}$ mod $(\sqrt{-7})$ for degree one, unramified $\fp\subset\ZZ_L$ with $2<N(\fp)<100$.}
\begin{center}
\begin{tabular}{l|l|l|l|l|l|l|l|l|l|l}\label{tbl:quadratic2}
       & 2 & 11 & 23 & 29 & 37 & 43 & 53 & 67 & 71 & 79\\ \hline
Steps 2a,b & 0.11 & 2.49 & 12.37 & 26.52 & 60.08 & 128.29 & 265.47 & 595.90 & 984.19 & 1561.67\\ \hline
Step 2c & 0.41 & 26.77 & 123.82 & 208.38 & 328.95 & 450.68 & 686.25 & 1431.75 & 1414.790 & 1774.52 \\ \hline
$\bar{a}_\fp$ & 0 & 0 & 0 & 3 & 0 & 3 & 0 & 0 & 3 & 0\\ \hline
$\bar{b}_\fp$ & 6 & 5 & 6 & 3 & 6 & 3 & 5 & 5 & 3 & 6
\end{tabular}
\end{center}
\label{default}
\end{table}%

Simultaneously diagonalizing the matrices of these Hecke operators we obtain two mod $(\sqrt{-7})$ systems of eigenvalues that we write $\bar{a}_\fp$ and $\bar{b}_\fp$.  We choose this notation because those systems are the reductions modulo 7 of the corresponding trivial weight systems $a_\fp$ and $b_\fp$ from earlier.  We have an explicit modulo 7 congruence between a nonendoscopic form in weight $(3,3,0)$ and an endoscopic form in weight $(0,0,0)$.  Thus, the modulo 7 Galois representation associated to the system $b_\fp$ is reducible.
\end{example}

\begin{example}
$U_{L/F}(3)$, $F=\QQ(\sqrt{13})$, $L=F\left(\sqrt{-13-2\sqrt{13}}\right)$, weight (0,0,0):

In this example, the class number of the principal genus is 9, as is the dimension of the corresponding space of   automorphic forms with trivial weight.
\begin{table}[htdp]
\caption{Computation of $T_{\fp,1}$ for degree one, unramified $\fp\subset\ZZ_L$ with $2<N(\fp)<250$.}
\begin{center}
\begin{tabular}{c|c|c|c|c|c|c|c|c|c|c|c}\label{tbl:quartic}
$N(\fp)$       & 29 & 53 & 61 & 79 & 107 & 113 & 131 & 139  & 157 & 191 &211\\ \hline
time (s) & 15.15 &51.73 & 70.35 & 123.21 & 216.82 & 242.20 & 339.50 & 378.81 & 486.22 & 727.81& 943.89\\ \hline
$\bar{a}_\fp $&4&11&12&5&5&3&6&10&1&11&2
\end{tabular}
\end{center}
\label{default}
\end{table}%
We computed the matrices of the Hecke operators acting on the $\QQ$-vector space $M(\QQ)$ for unramified, degree one $\fp\subset\ZZ_L$ with $2<N(\fp)<250$.  There appears to be a 1-dimensional ``Eisenstein'' subspace on which $T_{\fp,1}$ acts via $\deg T_{\fp,1}=N(\fp)^2 + N(\fp) + 1$.  The 8-dimensional complement of this line decomposes into $\QQ$-irreducible subspaces of dimensions 2, 2, and 4.  In all of these computations, the level subgroup is the stabilizer of the standard lattice $\ZZ_L^3\subset L^3$.

The Hecke algebra does acts nonsemisimply on the space $M(\FF_{13})$ of modulo 13 automorphic forms.  It appears that the minimal polynomial of $T_{\fp,1}$ has degree $6$ when $N(\fp)\equiv 1\pmod{13}$ and degree $7$ when $N(\fp)\equiv 3,9\pmod{13}$.  (Other residue classes do not occur for norms of degree one primes of $F$ splitting in $L$.)  When $N(\fp)\equiv 1\pmod{13}$, the eigenvalue $3\equiv N(\fp)^2+N(\fp)+1\pmod{13}$ occurs with multiplicity 5, while when $N(\fp)\equiv 3,9\pmod{13}$, the eigenvalue $0\equiv N(\fp)^2+N(\fp)+1\pmod{13}$ occurs with multiplicity 1.  Finally, is a 1-dimensional eigenspace in $M(\FF_{13})$ with eigenvalues $\bar{a}_q$ as in Table~\ref{tbl:quartic}.
\end{example}

\begin{example}
$U_{L/F}(3)$, $K=\QQ(\zeta_7+\zeta_7^{-1})$, $L=\QQ(\zeta_7)$, weight $(0,0,0)$:
\end{example}

In this case, the class number of the principal genus is 2, with the classes represented by the standard lattice $\Lambda_1=\ZZ_L^3$ and its $29$-neighbour $\Lambda_2$ with basis
\begin{multline*}
(\zeta_7+\zeta_7^4-\zeta_7^5,0,0),\quad (6,1,0),\\
 \tfrac{1}{29}(
-138- 234\zeta_7 - 210\zeta^2 - 303\zeta^3 - 258\zeta^4 - 117\zeta^5,\\ 
16 + 12\zeta + 13\zeta^2 + 20\zeta^3 + 11\zeta^4 + 6\zeta^5, \\
16 + 12\zeta + 13\zeta^2 + 20\zeta^3 + 11\zeta^4 + 6\zeta^5).
\end{multline*}
(This calculation took 1.85 seconds.)

\begin{table}[htdp]
\caption{Timings: computation of $T_{\fp,1}$ mod $(\sqrt{-7})$ for $\fp$ with $2<N(\fp)<100$ and $\fp\neq\bar{\fp}$}
\begin{center}
\begin{tabular}{c|c|c|c|c|c|c|c|c|c}\label{tbl:cyclo}
$N(\fp)$       & 29 & 43 & 71 & 113 & 127 & 197 & 211  & 239 & 281\\ \hline
time (s) & 2.16 &2.85 & 6.77 & 16.43 & 21.35 & 51.14 & 53.58 & 73.05 & 101.84\\ \hline
$a_\fp$ &871& 1893& 5113& 12883& 16257& 39007& 44733& 57361& 79243\\ \hline
$b_\fp$		& -25& 101 & 185 & -109 & 129 & 479 & -67 & 17 & 395
\end{tabular}
\end{center}
\label{default}
\end{table}%
Automorphism groups of $\Lambda_1$ and $\Lambda_2$ and their roles in this computation.
As above, $a_\fp=N(\fp)^2+N(\fp)+1=\deg T_{\fp,1}$, and the form with system of Hecke eigenvalues $a_\fp$ is a lift from $U(1)\times U(1)\times U(1)$.  Also, observe that
\[
a_\fp\equiv b_\fp \equiv 3\pmod{7},
\]
implying that the modulo 7 Galois representation attached to the system $b_\fp$ is reducible.
\begin{acknowledgement}
The authors would like to thank Lassina Demb\'el\'e and David Loeffler for helpful conversations.  
\end{acknowledgement}


\begin{thebibliography}{99}

\bibitem{Borel}
Armand Borel, \emph{Linear algebraic groups}, second enlarged ed., Graduate Texts in Math., vol.\ 126, Springer-Verlag, New York, 1991.

\bibitem{Magma} 
Wieb\ Bosma, John\ Cannon, and Catherine Playoust, \emph{The Magma algebra system. I. The user language}, J.~Symbolic Comput. \textbf{24} (1997), vol.~3--4, 235--265.

\bibitem{Chisholm}
Sarah Chisholm, \emph{Lattice methods for algebraic modular forms on quaternionic unitary groups}, Ph.D.\ thesis, University of Calgary, anticipated 2013.

\bibitem{CMT04}
Arjeh M.\ Cohen, Scott H.\ Murray, and D.\ E.\ Taylor, \emph{Computing in groups of Lie type}, Math.\ Comp. \textbf{73} (2004), no.\ 247, 1477--1498.

\bibitem{Cohen1}
Henri Cohen, \emph{A course in computational algebraic number theory}, Graduate Texts in Math., vol.~138, Springer-Verlag, Berlin, 1993.

\bibitem{Cohen2}
Henri Cohen, \emph{Advanced topics in computational algebraic number theory}, Graduate Texts in
Math., vol.~193, Springer-Verlag, Berlin, 2000.

\bibitem{CD}
Clifton Cunningham and Lassina Demb\'el\'e, \emph{Computation of genus 2 Hilbert-Siegel modular forms on $\mathbb{Q}(\sqrt{5})$ via the Jacquet-Langlands Correspondence}, Experimental Math.\ \textbf{18} (2009), no.\ 3, 337--345.

\bibitem{Dembele}
Lassina Demb\'el\'e, \emph{Quaternionic Manin symbols, Brandt matrices and Hilbert modular forms}, Math.\ Comp.\ \textbf{76} (2007), no.~258, 1039--1057. 

\bibitem{Dembele2}
Lassina Demb\'el\'e, \emph{A non-solvable Galois extension of $\Q$ ramified at $2$ only}, C.\ R.\ Acad.\ Sci.\ Paris, Ser.\ I, \textbf{347} (2009), 111--116.

\bibitem{DD}
Lassina Demb\'el\'e and Steve Donnelly, \emph{Computing Hilbert modular forms over fields with nontrivial class group}, Algorithmic number theory (Banff, 2008), Lecture Notes in Comput.\ Sci., vol.\ 5011, Springer, Berlin, 2008, 371--386.

\bibitem{DGV}
Lassine Demb\'el\'e, Matthew Greenberg, and John Voight, \emph{Nonsolvable number fields ramified only at $3$ and $5$}, Compositio Math.\ \textbf{149} (2011), no.\ 3, 716--734.

\bibitem{DV}
Lassina Demb\'el\'e and John Voight, \emph{Explicit methods for Hilbert modular forms}, accepted to Elliptic curves, Hilbert modular forms and Galois deformations.

\bibitem{dG01}
W.\ A.\ de Graaf, \emph{Constructing representations of split semisimple Lie algebras}, J.\ Pure Appl.\ Algebra, Effective methods in algebraic geometry (Bath, 2000), \textbf{164} (2001), no.\ 1--2, 87--107.

\bibitem{Dieulefait}
Luis Dieulefait, \emph{A non-solvable extension of $\Q$ unramified outside 7}, to appear in Compositio Math.

\bibitem{Eichler}
Martin Eichler, \emph{On theta functions of real algebraic number fields}, Acta Arith.\ \textbf{33} (1977), no.~3, 269--292.

\bibitem{Eichlerbook}
Martin Eichler, \emph{Quadratische Formen und orthogonale Gruppen}, Springer-Verlag, Berlin, 1952.

\bibitem{Eichlerbasis}
Martin Eichler, \emph{The basis problem for modular forms and the traces of the Hecke operators}, Modular functions of one variable, I (Proc.\ Internat.\ Summer School, Univ.\ Antwerp, Antwerp, 1972), Lecture Notes in Math., vol.~320, Springer, Berlin, 1973, 75--151. 

\bibitem{Eichlerbasis1}
Martin Eichler, \emph{Correction to: ``The basis problem for modular forms and the traces of the Hecke operators''}, Modular functions of one variable, IV (Proc.\ Internat.\ Summer School, Univ.\ Antwerp, Antwerp, 1972), Lecture Notes in Math., vol.~476, Springer, Berlin, 1975, 145--147.

\bibitem{FultonHarris}
William Fulton and Joe Harris, \emph{Representation theory: a first course}, Graduate Texts in Math., vol.\ 129, Springer-Verlag, New York, 1991.

\bibitem{GHY}
Wee Teck Gan, Jonathan Hanke, and Jiu-Kang Yu, \emph{On an exact mass formula of Shimura}, Duke Math.\ J.\ \textbf{107} (2001), no.\ 1, 103--133.

\bibitem{GY}
Wee Teck Gan and Jiu-Kang Yu, \emph{Group schemes and local densities}, Duke Math.\ J.\ \textbf{105} (2000), no.\ 3, 497--524.

\bibitem{GreenbergVoight}
Matt Greenberg and John Voight, \emph{Computing systems of Hecke eigenvalues associated to Hilbert modular forms}, Math.\ Comp.\ \textbf{80} (2011), 1071--1092.

\bibitem{Gross1}
Benedict Gross, \emph{Algebraic modular forms}, Israel J.\ Math.\ \textbf{113} (1999), 61--93.

\bibitem{Hoffmann}
Detlev W.\ Hoffmann, \emph{On positive definite Hermitian forms}, Manuscripta Math.\ \textbf{71} (1991), 399--429.

\bibitem{HPS}
Hiroaki Hijikata, Arnold K.\ Pizer, and Thomas R.\ Shemanske, \emph{The basis problem for modular forms on $\Gamma_0(N)$}, Amer.\ Math.\ Soc., Providence, 1989.

\bibitem{Humphreys}
James E.\ Humphreys, \emph{Linear algebraic groups}, Graduate Texts in Math., vol.\ 21, Springer-Verlag, New York, 1975.

\bibitem{Iyanaga1}
K.\ lyanaga, \emph{Arithmetic of special unitary groups and their symplectic representations}, J.\
Fac.\ Sci.\ Univ.\ Tokyo (Sec.\ 1) \textbf{15} (1968), no.\ 1, 35--69.

\bibitem{Iyanaga2}
K.\ Iyanaga, \emph{Class numbers of definite Hermitian forns}, J.\ Math.\ Soc.\ Jap.\ \textbf{21} (1969), 359--374.

% \bibitem{Iyanaga3}
% K.\ Iyanaga, \emph{Arithmetic of Hermitian forms and related forms}, J.\ Number Theory \textbf{4} (1972), 573--586.

% \bibitem{Jacobowitz}
% R.\ Jacobowitz, \emph{Hermitian forms over local fields}, American J.\ Math.\ \textbf{84} (1962), no.\ 3, 441--465. 

\bibitem{jacqlang} 
Herv\'e Jacquet and Robert P.~Langlands, \emph{Automorphic forms on GL(2)}, Lectures Notes in Math., vol.~114, Springer-Verlag, Berlin, 1970.

\bibitem{khare1}
C.~Khare and J.-P.~Wintenberger, \emph{On Serre's conjecture for 2-dimensional mod $p$ representations of $\Gal(\overline{\Q}/\Q)$}, Ann.\ of Math.\ (2) \textbf{169} (2009), no.\ 1, 229--253. 

\bibitem{KV}
Markus Kirschmer and John Voight, \emph{Algorithmic enumeration of ideal classes for quaternion orders}, SIAM J.\ Comput.\ (SICOMP) \textbf{39} (2010), no.\ 5, 1714--1747.

\bibitem{Knus}
Max-Albert Knus, \emph{Quadratic and Hermitian forms over rings}, Springer-Verlag, Berlin, 1991.

\bibitem{Kn3}
Martin Kneser, \emph{Klassenzahlen indefiniter quadratischer Formen in drei oder mehr Ver\"anderlichen}, Arch.\ Math.\ \textbf{7} (1956), 323--332.

\bibitem{Kneser}
Martin Kneser, \emph{Klassenzahlen definiter quadratischer Formen}, Arch. Math. \textbf{8} (1957), 241--250. 

\bibitem{Kneser-strong}
Martin Kneser, \emph{Strong approximation}, Algebraic groups and discontinuous subgroups (Proc.\ Sympos.\ Pure Math., Boulder, Colo., 1965), American Mathematical Society, Providence, 1966, 187–196.

\bibitem{kohel}
David Kohel, \emph{Hecke module structure of quaternions}, Class field theory: its centenary and prospect (Tokyo, 1998), ed.\ K.\ Miyake, Adv.\ Stud.\ Pure Math., vol.~30, Math.\ Soc.\ Japan, Tokyo, 2001, 177--195.

\bibitem{LanskyPollack}
Joshua Lansky and David Pollack, \emph{Hecke algebras and automorphic forms}, Compositio Math.\ \textbf{130} (2002), no.~1, 21--48. 

\bibitem{vLCL92}
M.A.A.\ van Leeuwen, A.M.\ Cohen, and B.\ Lisser, \emph{LiE, a package for Lie group computations},  CAN, Amsterdam, 1992.

\bibitem{Loeffler}
David Loeffler, \emph{Explicit calculations of automorphic forms for definite unitary groups}, LMS J.\ Comput.\ Math.\ \textbf{11} (2008), 326--342. 

\bibitem{KirschmerVoight}
Markus Kirschmer and John Voight, \emph{Algorithmic enumeration of ideal classes for quaternion orders}, SIAM J.\ Comput.\ (SICOMP) \textbf{39} (2010), no.\ 5, 1714--1747.

\bibitem{OMeara}
O.\ Timothy O'Meara, \emph{Introduction to quadratic forms}, Springer-Verlag, Berlin, 2000.

\bibitem{Pizer}
Arnold Pizer, \emph{An algorithm for computing modular forms on $\Gamma_0(N)$}, J.\ Algebra \textbf{64} (1980), vol.~2, 340--390.

\bibitem{PS}
Wilhelm Plesken and Bernd Souvignier, \emph{Computing isometries of lattices}, Computational algebra and number theory (London, 1993), J.\ Symbolic Comput.\ \textbf{24} (1997), no.~3--4, 327--334. 

\bibitem{Scharlau}
Winfried Scharlau, \emph{Quadratic and Hermitian forms}, Springer-Verlag, Berlin, 1985.

\bibitem{Schiemann}
Alexander Schiemann, \emph{Classification of Hermitian forms with the neighbour method}, J.\ Symbolic Comput.\ \textbf{26} (1998), no.\ 4, 487–-508. 

\bibitem{SH}
Rudolf Scharlau and Boris Hemkemeier, \emph{Classification of integral lattices with large class number}, Math.\ Comp.\ \textbf{67} (1998), no.~222, 737--749. 

\bibitem{Schulze-Pillot}
Rainer Schulze-Pillot, \emph{An algorithm for computing genera of ternary and quaternary quadratic
forms}, Proc.\ Int.\ Symp.\ on Symbolic and Algebraic Computation, Bonn, 1991.

\bibitem{Shimura-unitary}
Goro Shimura, \emph{Arithmetic of unitary groups}, Ann.\ of Math.\ (2) \textbf{79} (1964), 369--409.

\bibitem{SocratesWhitehouse}
Jude Socrates and David Whitehouse, \emph{Unramified Hilbert modular forms, with examples relating to elliptic curves}, Pacific J.\ Math.\ \textbf{219} (2005), no.~2, 333--364. 

\bibitem{Sage}
William Stein, \emph{SAGE Mathematics Software} (version 3.1.1), The SAGE Group, 2008, \\ \texttt{http://www.sagemath.org/}.

\bibitem{VoightMaxOrder}
John Voight, \emph{Identifying the matrix ring: algorithms for quaternion algebras and quadratic forms}, accepted, \verb|arXiv:1004.0994|.

\bibitem{Weyl}
Hermann Weyl, \emph{The classical groups: their invariants and representations}, Princeton University, Princeton, 1966.

\end{thebibliography}
\end{document}